\documentclass[12 pt]{amsart}

\usepackage{etex,seqsplit}
\usepackage[shortlabels]{enumitem}
\usepackage{amsmath}
\usepackage{amscd}
\usepackage{amsthm}
\usepackage{amsfonts}
\usepackage{amssymb}
\usepackage{eucal}
\usepackage{mathrsfs, fancyvrb}
\usepackage{morefloats, comment}
\usepackage{thm-restate}
\usepackage{hyperref}
\usepackage[margin=1in]{geometry}
\usepackage{tikz}
\usepackage{tikz-cd, alltt, changepage}
\usepackage{booktabs}
\usepackage{colortbl}
\usepackage{xcolor}
\usepackage{xfrac}

\usepackage{todonotes}

\theoremstyle{plain}
\newtheorem{theorem}{Theorem}[section]
\newtheorem{proposition}[theorem]{Proposition}
\newtheorem{lemma}[theorem]{Lemma}
\newtheorem{corollary}[theorem]{Corollary}
\newtheorem{conjecture}[theorem]{Conjecture}

\theoremstyle{definition}
\newtheorem{definition}[theorem]{Definition}

\newtheorem*{theorem*}{Theorem}
\newtheorem*{proposition*}{Proposition}
\newtheorem*{lemma*}{Lemma}

\theoremstyle{remark}
\newtheorem{remark}[theorem]{Remark}

\numberwithin{equation}{section}

\newcommand{\FF}{\mathbb F}
\newcommand{\QQ}{\mathbb Q}
\newcommand{\OO}{\mathcal O}
\newcommand{\ZZ}{\mathbb Z}

\newcommand{\on}{\operatorname}
\newcommand{\ol}{\overline}
\newcommand{\wt}{\widetilde}
\newcommand{\wh}{\widehat}
\makeatletter
\newcommand*{\defeq}{\mathrel{\rlap{%
                     \raisebox{0.3ex}{$\m@th\cdot$}}%
                     \raisebox{-0.3ex}{$\m@th\cdot$}}%
                     =}
\makeatother

\newcommand{\id}{\on{id}}
\newcommand\ab{\on{ab}}
\newcommand{\Hom}{\on{Hom}}

\newcommand{\hyperell}{\mathscr{X}}
\newcommand{\hyperellSource}{\mathscr{C}}

\newcommand\hyperelliptic[2]{\mathscr X_{#1, #2}} 
\newcommand\hyperellipticSource[2]{\mathscr C_{#1, #2}} 
\newcommand\grouphyp[2]{\widetilde{\mathrm{G} \mathcal {S}}_{2{#1}+2, #2}} 
\newcommand\grouphypboth[3]{\widetilde{\mathrm{G} \mathcal {S}}_{2{#1}+2-{#2}, #3}} 

\newcommand\standardFamily[3]{\mathscr Y_{#1, #3}^{(#2)}} 
\newcommand\standardTarget[3]{\mathscr W_{#1, #3}^{(#2)}} 

\newcommand{\gify}[2]{\mathrm{G}\Phi_{#1 \to #2}}
\newcommand{\fy}[2]{\Phi_{#1 \to #2}}

\newcommand{\hyphat}{\wt{\mathcal{S}}} 

\newcommand{\hypaddict}{\mathcal S} 

\newcommand\bc{{\mathbb C}}

\newcommand\bq{{\mathbb Q}}

\newcommand\bz{{\mathbb Z}}

\newcommand\ba{{\mathbb A}}

\newcommand{\mc}{\mathcal}

\newcommand{\zh}{\widehat{\mathbb Z}}

\DeclareMathOperator\Sp{Sp}
\newcommand\mono{H}
\DeclareMathOperator\GSp{GSp}
\DeclareMathOperator\GL{GL}
\DeclareMathOperator\Gal{Gal}

\DeclareMathOperator\mult{mult}
\DeclareMathOperator\disc{disc}

\newcommand \ra{\rightarrow}
\newcommand{\mf}{\mathfrak}

\DeclareMathOperator\charpoly{ch}
\newcommand{\frob}{\operatorname{Frob}}

\makeatletter
\def\@bignumber#1#2{%
  \ifx#2\end
    #1\let\next\@gobble
  \else
    #1\hspace{0pt plus 1pt}\let\next\@bignumber
  \fi
  \next#2}
\newcommand{\bignumber}[1]{\@bignumber#1\end}
\makeatother

\newcommand{\spec}{\on{Spec}}
\newcommand{\cyc}{\on{cyc}}

\makeatletter
\newcommand{\customlabel}[2]{%
   \protected@write \@auxout {}{\string \newlabel {#1}{{#2}{\thepage}{#2}{#1}{}} }%
   \hypertarget{#1}{#2}
}

\title[Hyperelliptic Curves with Maximal Galois Action on Torsion of Jacobian]{Hyperelliptic Curves with Maximal Galois Action on the Torsion Points of their Jacobians}

\date{\today}

\author[Aaron Landesman]{Aaron Landesman}

\author[Ashvin A. Swaminathan]{Ashvin A. Swaminathan}

\author[James Tao]{James Tao}

\author[Yujie Xu]{Yujie Xu}

\AtEndDocument{\bigskip{\footnotesize%
  \textsc{Department of Mathematics, Stanford University, \mbox{Stanford, CA 94305}} \par
  \textit{E-mail address}, Aaron Landesman: \texttt{aaronlandesman@stanford.edu} \par
  \addvspace{\medskipamount}
  \textsc{Department of Mathematics, Princeton University, \mbox{Princeton, NJ 08544}} \par
  \textit{E-mail address}, Ashvin A. Swaminathan: \texttt{ashvins@math.princeton.edu} \par
  \addvspace{\medskipamount}
  \textsc{Department of Mathematics, Harvard College, \mbox{Cambridge, MA 02138}} \par
\textit{E-mail address}, James Tao: \texttt{jamestao@college.harvard.edu}
\par
  \addvspace{\medskipamount}
  \textsc{Department of Mathematics, Harvard University, \mbox{Cambridge, MA 02138}} \par
  \textit{E-mail address}, Yujie Xu: \texttt{yujiex@math.harvard.edu}
}}

\begin{document}

\begin{abstract}
In this article, we show that in each of four standard families of hyperelliptic curves, there is a density-$1$ subset of members with the property that their Jacobians have adelic Galois representation with image as large as possible. This result constitutes an explicit application of a general theorem on arbitrary rational families of abelian varieties to the case of families of Jacobians of hyperelliptic curves.
Furthermore, we provide explicit examples
of hyperelliptic curves of genus $2$ and $3$ over $\mathbb Q$
whose Jacobians have such maximal adelic Galois representations.
\end{abstract}

\maketitle

\section{Introduction}
\subsection{Background}
\label{subsection:intro-background}
Let $A$ be a principally polarized abelian variety (PPAV) of dimension $g \geq 1$
over a number field $K$.
Fix an algebraic closure $\ol{K}$ of $K$, and let $G_K \defeq \Gal(\ol{K}/K)$ be the absolute Galois group.
The action of $G_K$ on the torsion points of $A(\ol{K})$ gives rise to the
\emph{adelic} Galois representation
$$\rho_A \colon G_K \to\GSp_{2g}(\wh{\ZZ}).$$
For prime numbers $\ell$, the \emph{mod-$\ell$ Galois representation} $\rho_{A,\ell} \colon G_K \to \GSp_{2g}(\bz/\ell\bz)$ is defined by reducing the image of $\rho_A$ modulo $\ell$.
See~\cite[Section 2.2]{seaweed} and
\cite[Section 3.1]{landesman-swaminathan-tao-xu:rational-families} for more detailed descriptions of these representations.

In 1972, Jean-Pierre Serre proved the celebrated Open Image Theorem (see~\cite{causalrelationship}), which states that for an elliptic curve $E/K$ without complex multiplication, $\rho_E(G_K)$ is an open subgroup of, and hence has finite index in, the profinite group $\GSp_{2}(\wh{\ZZ})$.
While the Open Image Theorem implies that the adelic Galois representation maps onto a large subgroup of $\GSp_{2g}(\wh{\ZZ})$, the image of this representation is not always equal to $\GSp_{2g}(\wh{\ZZ})$.
Indeed, Serre observed in~\cite[Proposition 22]{causalrelationship} that for every elliptic curve $E/\QQ$, the image of $\rho_E$ has even index in $\GSp_2(\wh{\ZZ})$.
Nonetheless, in~\cite[Sections 5.5.6-8]{causalrelationship},
Serre constructs several examples of elliptic curves over $\bq$
whose Galois representations have ``maximal image'' among all elliptic curves, in the sense that the index of the \mbox{image in $\GSp_2(\widehat{\mathbb Z})$ is equal to $2$.}

The obstruction faced by elliptic curves over $\QQ$ to having surjective adelic Galois representation no longer exists when $\QQ$ is replaced by a larger number field.
In~\cite{greasy}, Greicius constructs an example of an elliptic curve over a cubic extension of $\bq$ whose Galois representation
has image equal to $\GSp_2(\widehat{\mathbb Z})$.
Furthermore, in~\cite{seaweed}, Zywina constructs an example of a non-hyperelliptic curve of genus $3$ over $\bq$ whose Jacobian has adelic Galois image equal to $\GSp_6(\widehat{\mathbb Z})$.
While there are explicit examples in genera $1$ and $3$, to the authors' knowledge,
there are no examples in the literature of curves of genus
$2$ with associated Galois representation having maximal image among such curves.
Additionally, there are no known examples of hyperelliptic curves of genus $3$ whose Galois image is maximal.
Nevertheless, there are a few examples that come close:
	In~\cite[Theorem 5.4]{dooleyfat},
	Dieulefait gives an example of a
	genus-$2$ curve over $\mathbb Q$
	whose Jacobian has mod-$\ell$
	monodromy equal to $\GSp_4(\mathbb Z/\ell \mathbb Z)$
	for $\ell \ge 5$.
	Similarly, in~\cite[Corollary 1.1]{anni2016residual},
	an example of a genus-$3$ hyperelliptic curve
	over $\bq$ whose Jacobian has mod-$\ell$ Galois image equal to $\GSp_6(\mathbb Z/\ell \mathbb Z)$
	for primes $\ell \geq 3$ is constructed.
	However, in both of these cases, it is easy to check
	that these examples have mod-2
	Galois image that is not maximal
	among all hyperelliptic curves of genus
	$2$ or $3$.
	In Theorem~\ref{exemplinongratia},
	we improve on the results of~\cite{dooleyfat} and~\cite{anni2016residual},
	giving explicit examples of hyperelliptic
	curves of genus $2$ and $3$ over $\QQ$ with maximal adelic Galois image.
	The reader may wish to also refer to the related recent paper~\cite{anni2017constructing}, which constructs
	hyperelliptic curves with maximal mod-$\ell$ Galois image in all genera $g$ with the property that $2g+2$ can be expressed as of sum of two primes in two different ways, with none of the primes being the largest prime less than $2g + 2$.\footnote{Note that~\cite{anni2017constructing}
	therefore does not address the cases $g = 2, 3$,
	which we cover in this paper.}

In addition to finding explicit examples of PPAVs with maximal Galois image, there are a number of results in the literature concerning how many members of a given family of PPAVs have maximal adelic Galois image.
The first key result in this direction is due to Duke, who proved in~\cite{duke:elliptic-curves-with-no-exceptional-primes} that ``most'' elliptic curves $E/\QQ$ in the standard family with Weierstrass equation $y^2 = x^3 + ax + b$ have the property that $\rho_{E,\ell}(G_{\QQ}) = \GSp_2(\ZZ/\ell \ZZ)$ for every prime number $\ell$; here, the term ``most'' means a density-$1$ subset of curves ordered by na\"{i}ve height. Building upon the work of Duke, Jones proved in~\cite[Theorem 4]{josofabank} that $[\GSp_{2g}(\wh{\ZZ}) : \rho_E(G_K)] = 2$ for most elliptic curves $E$ in the standard family over $\QQ$.
In~\cite[Theorem 1.15]{zywina2010hilbert}, Zywina generalized
the above results, showing that most members of every non-isotrivial rational family of elliptic curves over an arbitrary number field have maximal adelic Galois image, subject to the constraints that arise from the arithmetic and geometric properties of the family under consideration.
Additional results over $\mathbb Q$ were
obtained in~\cite{grant:a-formula-for-the-number-of-elliptic-curves-with-exceptional-primes},~\cite{cojocaruH:uniform-results-for-serres-theorem-for-elliptic-curves}, and~\cite{cojocaruGJ:one-parameter-families-of-elliptic-curves}
(see~\cite[p.~6]{zywina2010hilbert} for a more detailed overview).
In Theorem~\ref{mainbldg}, we give an explicit version of~\cite[Theorem 1.1]{landesman-swaminathan-tao-xu:rational-families} -- a result that generalizes Zywina's results to rational families of higher-dimensional PPAVs --
for many common families of hyperelliptic curves.
This yields a generalization of \cite[Theorem 1.2]{zywina2010elliptic}
and~\cite[Theorem 4]{josofabank} to hyperelliptic curves of higher genus.

\subsection{Main Results}

In this paper, we primarily consider those PPAVs that arise as Jacobians of hyperelliptic curves belonging to one of the following four standard families; we restrict our consideration to curves of genus at least $2$ because the results of Zywina in~\cite{zywina2010hilbert} completely handle the case of elliptic curves.
\begin{definition}
	\label{definition:standard-families}
Let $g \geq 2$ be an integer, and for $i \in \{1, 2, 3, 4 \}$ define $\standardTarget g i K$ by
\begin{align*}
&\standardTarget g 1 K = \mathbb A^{2g+1}_{[a_0, \ldots, a_{2g}]} \setminus \Delta^{(1)}, \,\,\,\,\, \standardTarget g 2 K = \mathbb A^{2g+2}_{[a_0, \ldots, a_{2g+1}]}\setminus \Delta^{(2)}, \\
&\,\,\,\,\,  \standardTarget g 3 K = \mathbb A^{2g}_{[a_0, \ldots, a_{2g-1}]}\setminus \Delta^{(3)}, \,\,\,\,\,  \standardTarget g 4 K = \mathbb A^{2g+1}_{[a_0, \ldots, a_{2g}]} \setminus \Delta^{(4)},
\end{align*}
where each $\Delta^{(i)}$ is the discriminant locus, on which the indicated polynomial has at least one multiple root:
\[
\begin{array}{cc}
x^{2g+1} + a_{2g}x^{2g} + \cdots + a_0  &\rightsquigarrow  \Delta^{(1)}  \\
x^{2g+2} + a_{2g+1}x^{2g+1} + \cdots + a_0 & \rightsquigarrow  \Delta^{(2)}\\
x^{2g+1} + a_{2g-1}x^{2g-1} + \cdots + a_0  &\rightsquigarrow \Delta^{(3)}\\
x^{2g+2} + a_{2g}x^{2g} + \cdots + a_0 &\rightsquigarrow \Delta^{(4)}.
\end{array}
\]
Consider the following vanishing loci, and view them as families over $\standardTarget g i K$ via projection onto the first factor:
\begin{align*}
V(y^2 - x^{2g+1} - a_{2g}x^{2g} - \cdots - a_0) &\hookrightarrow \standardTarget g 1 K \times \mathbb A^2_{[x,y]} \to \standardTarget g 1 K, \\
V(y^2 - x^{2g+2} - a_{2g+1}x^{2g+1} - \cdots - a_0) &\hookrightarrow \standardTarget g 2 K \times \mathbb A^2_{[x,y]} \to \standardTarget g 2 K,  \\
V(y^2 - x^{2g+1} - a_{2g-1}x^{2g-1} - \cdots - a_0) &\hookrightarrow \standardTarget g 3 K\times \mathbb A^2_{[x,y]} \to  \standardTarget g 3 K, \\
V(y^2 - x^{2g+2} - a_{2g}x^{2g} - \cdots - a_0) &\hookrightarrow \standardTarget g 4 K \times \mathbb A^2_{[x,y]} \to \standardTarget g 4 K.
\end{align*}
For $1 \leq i \leq 4$, define $\standardFamily g i K$, the \emph{standard families} of genus-$g$ hyperelliptic curves by completing the above smooth affine
curve over $\standardTarget g i K$ to a smooth projective curve over $\standardTarget g i K$. The definition of $\Delta^{(i)}$ ensures that these are indeed genus-$g$ hyperelliptic curves.
For a $K$-valued point $u \in \standardTarget g i K (K)$, we denote by $A_u$ the Jacobian (which is necessarily a $g$-dimensional PPAV) of the fiber over $u$ of the corresponding standard family.
\end{definition}

As we show in Section~\ref{symbed},
the mod-$2$ Galois image of the Jacobian of a member of $\standardFamily{g}{i}{K}$ always lands in a certain copy of the symmetric group $S_{2g+2-(i \bmod 2)} \subset \Sp_{2g}(\ZZ/2\ZZ)$. Denote by $\grouphypboth{g}{(i \bmod 2)}{K}$ the intersection of the following two subgroups of $\GSp_{2g}(\wh{\ZZ})$: (1) the subgroup of those matrices with multiplier landing in $\chi(K) \subset \wh{\ZZ}^\times$, where $\chi$ denotes the cyclotomic character, and (2) the preimage of $S_{2g+2-(i \bmod 2)}$ under the projection map $\GSp_{2g}(\wh{\ZZ}) \to \Sp_{2g}(\ZZ/2\ZZ)$.
Let $\on{Ht}\colon\mathbb P^r(\overline K) \rightarrow \mathbb R_{>0}$
denote the absolute multiplicative height on projective space, and define a height function $\| - \|$ on the lattice $\mathcal O^r_K$
sending $\left( t_1, \ldots, t_r \right) \mapsto \max_{\sigma,i}|\sigma(t_i)|$,
where $\sigma$ varies over all field embeddings $\sigma\colon K \hookrightarrow \mathbb C$.
Having fixed this notation, our first main theorem may be stated as follows:

\begin{theorem}\label{mainbldg}	
Let $B> 0$, $i \in \{1,2,3,4\}$, $g \geq 2$, and let $n$ be an arbitrarily positive integer. Let $\delta_\QQ = 2$, and let $\delta_K = 1$ for $K \neq \QQ$. Then $[ \grouphypboth g {(i \bmod 2)}{K} : \rho_{A_u}(G_K) ] \geq \delta_K$
for all $u \in \standardFamily g i K(K)$,
and we have the following asymptotic statements, with the implied constants depending only on $n$, $g$, and $K$:
			\[
				\frac{|\{u \in \standardTarget g i K(\OO_K) : \lVert u \rVert \le B,\, [ \grouphypboth g {(i \bmod 2)}{K} : \rho_{A_u}(G_K) ] = \delta_K \}|}{|\{u \in \standardTarget g i K(\OO_K): \lVert u \rVert \le B\}|} = 1 + O((\log B)^{-n}),
			\]
\[
	\frac{|\{u \in \standardTarget g i K(K) : \on{Ht}(u) \leq B,\, [ \grouphypboth g {(i \bmod 2)}{K} : \rho_{A_u}(G_K) ] = \delta_K \}|}{|\{u \in \standardTarget g i K(K) : \on{Ht}(u) \le B\}|} = 1 + O((\log B)^{-n}).
			\]
\noindent Furthermore, the statement above applies
if we take $i = 2$ and replace
$\standardTarget g 2 K$ by any rational family
of hyperelliptic curves
dominating the moduli of hyperelliptic curves,
so long as the map to the moduli of hyperelliptic
curves has geometrically connected generic fiber.
\end{theorem}

The methods employed to prove density-$1$ results like Theorem~\ref{mainbldg} do not lend themselves well to the construction of explicit examples, which may be useful insofar as they can provide evidence in support of related conjectures.
We now give two explicit examples, improving upon the examples of \cite{dooleyfat} and \cite{anni2016residual} mentioned in Section~\ref{subsection:intro-background}.
To the authors' knowledge,
these are the first examples of hyperelliptic curves in genus $g = 2$ and $3$
whose mod-$\ell$ monodromy is equal to
$\GSp_{2g}(\bz/\ell\bz)$ when $\ell > 2$,
and equal to $S_{2g+2}$ when $\ell = 2$.
Moreover, we show the Galois representations of
these curves have index $2$ in the group
$\grouphyp g {\mathbb Q}$. Note that all hyperelliptic
curves over $\bq$ have Galois representation
strictly contained in $\grouphyp g {\mathbb Q}$,
as follows from
Corollary~\ref{monodromy-stack} (since the monodromy of any curve
is contained in that of the universal family).
Hence, our examples yield curves with
maximal monodromy among all hyperelliptic curves
of genus $2$ and $3$.

\begin{theorem}\label{exemplinongratia}
Let $C_2$ and $C_3$ over $\QQ$ be smooth projective models of the affine plane curves cut out by the equations
\begin{align*}
C_2 & : \quad y^2 = x^6 + 7471225x^5 + 16548721x^4 + 6639451x^3 + 16857421x^2 + \\
& \qquad \qquad \qquad 20754195x +9508695, \text{ and} \\
C_3 & : \quad y^2 = x^8 + 10781051650x^7 + 5302830080x^6 + 33362176x^5 + 10656581376x^4  + \\
& \qquad \qquad \qquad 5522318080x^3 + 4238752256x^2  + 3613465600x  + 3725404480.
\end{align*}
Then for each $g \in \{2, 3\}$, the Jacobian $J_{C_g}$ of $C_g$ is a $g$-dimensional PPAV over $\QQ$ satisfying the condition $[\grouphyp g {\mathbb Q} : \rho_{J_{C_g}}(G_\QQ)] = 2$. \end{theorem}
\begin{remark}
In checking the examples declared in Theorem~\ref{exemplinongratia}, we combined the methods developed in~\cite{anni2017constructing} and~\cite{seaweed} to expedite the verification process. It is also possible to modify the techniques introduced in~\cite{seaweed} to show that the curves cut out by the equations
\begin{align}
f(x) & = x^6 - 2x^4 - 2x^3 - 3x^2 - 2x + 1 \quad \text{and} \\
f(x) & = x^8 -4x^3 + 4x + 4,
\end{align}
which are of respective genera $2$ and $3$, both have maximal monodromy. The reader may contact any one of the authors if further details of the proof of this claim are desired.
\end{remark}

The rest of this paper is organized as follows:
Section~\ref{section:group-theory} is concerned with proving the
group-theoretic Theorem~\ref{theorem:small-ab}.
In Section~\ref{prelimswine},
we use Theorem~\ref{theorem:small-ab} to prove
Lemma~\ref{theorem:r=2}, which
is employed in the proof of
Theorem~\ref{mainbldg} to verify the claimed value
of $\delta_K$.
In Section~\ref{subsection:monodromy-of-families}
we compute the monodromy of various families of hyperelliptic curves.
We combine these two results to prove Theorem~\ref{mainbldg}
in Section~\ref{32isnotagoodscoreline}.
Finally, in Section~\ref{verified},
we prove Theorem~\ref{exemplinongratia}.

\section{Definitions and Properties of Symplectic Groups}
\label{section:group-theory}

This section is devoted to proving Theorem~\ref{theorem:small-ab},
which is needed for proving the main results of this paper, Theorems~\ref{mainbldg} and~\ref{exemplinongratia}. We start in Sections~\ref{subsection:stimpy} and~\ref{subsection:notation} by defining symplectic groups, discussing their basic properties, and introducing some recurring notation. Then, in Section~\ref{mygawdjamesyousabeast} we prove a result that is a crucial input to Section~\ref{section:proof-of-mainbldg}, where we prove Theorem~\ref{mainbldg}.
The reader may choose to continue directly to Section~\ref{section:proof-of-mainbldg} after studying the statement of Theorem~\ref{mygawdjamesyousabeast}.

\subsection{Symplectic Groups}\label{subsection:stimpy}

Let $R$ be a commutative ring, and let $g$ be a positive integer. Let $M$ be a free $R$-module of rank $2g$, and let $\langle -, - \rangle \colon M \times M \to R$ be a non-degenerate alternating bilinear form on $M$. Define the {\it general symplectic group} (otherwise known as the \emph{group of symplectic similitudes}) $\GSp(M) \subset \on{GL}(M)$ to be the subgroup of all $R$-automorphisms $S$ such that there exists some $m_S \in R^\times$, called the {\it multiplier} of $S$, satisfying $\langle S v, Sw \rangle = m_S \cdot \langle v, w \rangle$ for all $v, w \in M$. If $m_S$ exists, then it is necessarily unique, and one easily checks that the resulting {\it mult} map
\begin{align*}
	\mult \colon \GSp(M) & \rightarrow R^\times \\
	S & \mapsto m_S
\end{align*}
is a group homomorphism; we call its kernel the {\it symplectic group} $\Sp(M)$.

Choose an $R$-basis for $M$, and denote by $\Omega_{2g}$ the matrix which expresses the inner product $\langle - , - \rangle$ with respect to this basis.
The choice of basis gives rise to an identification $\GL(M) \simeq \GL_{2g}(R)$, and we take $\GSp_{2g}(R)$ to be the image of $\GSp(M)$ and $\Sp_{2g}(R)$ to be the image of $\Sp(M)$ under this identification. Let $\det \colon \GL_{2g}(R) \to R^\times$ be the determinant map, and observe that the diagram
\begin{center}
\begin{tikzcd}
\GSp(M) \arrow{r}{\sim} \arrow[swap]{rd}{\on{mult}^g} &  \GSp_{2g}(R) \arrow{d}{\on{det}} \\
& R^\times
\end{tikzcd}
\end{center}

\noindent commutes. Note
that $\GSp_{2g}(R) \subset \GL_{2g}(R)$ is the subgroup of all invertible matrices $S$ satisfying $S^T \Omega_{2g} S = (\on{mult} S) \, \Omega_{2g}$
and that $\Sp_{2g}(R) = \ker(\on{mult} \colon \GSp_{2g}(R) \to R^\times)$.

Let $\on{Mat}_{2g \times 2g}(R)$ be the space of $2g \times 2g$ matrices having entries in $R$, and consider the Lie algebras $\mf{gsp}_{2g}(R)$ and $\mf{sp}_{2g}(R)$ defined by
		\begin{align*}
        \mf{gsp}_{2g}(R) &\defeq \{\Lambda \in \on{Mat}_{2g \times 2g}(R) : \Lambda^T \Omega_{2g} + \Omega_{2g} \Lambda = d \cdot \Omega_{2g} \text{ for some } d \in R \}, \\
			\mf{sp}_{2g}(R) &\defeq \{\Lambda \in \on{Mat}_{2g \times 2g}(R) : \Lambda^T \Omega_{2g} + \Omega_{2g} \Lambda = 0 \}.
		\end{align*}
When studying Galois representations associated to PPAVs, we usually take $R$ to be one of the following: the profinite completion $\wh{\ZZ}$ of $\ZZ$, the ring of $\ell$-adic integers $\ZZ_{\ell}$ for a prime number $\ell$, or the finite cyclic ring $\ZZ / m \ZZ$ for a positive integer $m$.
Observe that we have the following isomorphisms of topological groups:
\begin{equation}\label{orientation1}
\GSp_{2g}(\ZZ_\ell)   \simeq  \varprojlim_k \GSp_{2g}(\ZZ/ \ell^k \ZZ) \quad \text{and}
\end{equation}
\begin{equation}\label{orientation2}
\prod_{\text{prime } \ell} \GSp_{2g}(\ZZ_\ell) \simeq  \GSp_{2g}(\wh{\ZZ}) \simeq  \varprojlim_m \GSp_{2g}(\ZZ / m \ZZ).
\end{equation}
The isomorphisms \eqref{orientation1} and~\eqref{orientation2} remain valid if $\GSp_{2g}$ is replaced by $\Sp_{2g}$. As for the Lie algebras, note that by sending $\Lambda \mapsto \id_{2g} + \ell^k \Lambda$ we obtain group isomorphisms
\begin{align*}
\mf{gsp}_{2g}(\ZZ/ \ell \ZZ) &\simeq \ker(\GSp_{2g}(\ZZ/\ell^{k+1} \ZZ) \to \GSp_{2g}(\ZZ/ \ell^k \ZZ)), \\
\mf{sp}_{2g}(\ZZ/ \ell \ZZ) &\simeq \ker(\Sp_{2g}(\ZZ/\ell^{k+1} \ZZ) \to \Sp_{2g}(\ZZ/ \ell^k \ZZ))
\end{align*}
for every $k \geq 1$, so when it is useful or convenient, we will sometimes use the Lie algebra notation to denote the above kernels.

\subsection{Computing Commutators of Large Subgroups of $\GSp_{2g}(\ZZ_2)$} \label{mygawdjamesyousabeast}
The objective of this section is to prove a soon-to-be-useful theorem concerning the commutator of a subgroup of $\GSp_{2g}(\ZZ_2)$ which is the preimage (under mod-2 reduction) of a subgroup of $\Sp_{2g}(\bz / 2 \bz)$ that contains a copy of the symmetric group $S_{2g+1}$.

\subsubsection{Embedding the Symmetric Group, Take 1} \label{take1}

We asserted in the discussion immediately preceding the statement of Theorem~\ref{mainbldg} that the symmetric group $S_{2g+1}$ may be viewed as a subgroup of $\Sp_{2g}(\ZZ/2\ZZ)$. We now provide a working description of the way in which this embedding is constructed; the manner in which this description applies to the context of studying hyperelliptic curves is discussed in Section~\ref{symbed}.
\begin{lemma} \label{lemma:include-s}
For every $g \ge 2$, we have an inclusion $S_{2g+2} \hookrightarrow \Sp_{2g}(\bz /2 \ZZ)$.
When $g = 2$, this inclusion is an isomorphism.
\end{lemma}
\begin{proof}
	Let $V$ be a $(2g+2)$-dimensional vector space over $\FF_2$, and equip $V \simeq \mathbb{F}_2^{2g+2}$ with the standard inner product. Let $t \defeq (1, \ldots, 1)$ be the vector whose components are all equal to $1$. Then the hyperplane $t^\perp \subset V$ of all vectors orthogonal to $t$ actually contains $t$ since $\dim V = 2g+2$ is even. Moreover, if we define $W = t^\perp / \on{span}(t)$, the inner product on $V$ descends to a nondegenerate alternating bilinear form on $W$. The action of $S_{2g+2}$ given by permuting the coordinates of $V$ fixes both $t$ and $t^\perp$, so it descends to an action on $W$ that preserves the bilinear form.
Thus, we obtain an inclusion of $S_{2g+2}$ into the group of symplectic transformations of $W$ with multiplier $1$. For a more conceptual
explanation of this inclusion in terms of the two-torsion of hyperelliptic
curves, see Section \ref{symbed}.
Upon choosing a suitable basis for $W$ we may identify this group with $\Sp_{2g}(\ZZ / 2 \ZZ)$. For $g = 2$, the resulting inclusion is an isomorphism because $\#(S_6) = 720 = \#(\Sp_4(\ZZ / 2 \ZZ))$.
\end{proof}
We embed $S_{2g+1} \hookrightarrow S_{2g+2}$ as the subgroup fixing the vector $(0,\dots,0,1) \in \FF_2^{2g+2}$.

\subsubsection{Notation}
\label{subsection:notation}

In what follows, we shall (for the most part) study subquotients of $\GSp_{2g}(\ZZ_2)$ and $\GSp_{2g}(\ZZ/ 2^k \ZZ)$ for $k$ a positive integer. We employ the following notational conventions:
\begin{itemize}
\item Let $H\subset \GSp_{2g}(\ZZ_2)$ be a closed subgroup.
\item For $m, n \in \ZZ_{> 0} \cup \{\infty\}$
		with $m > n$, let $\gify{2^m}{2^n} \colon \GSp_{2g}(\ZZ/2^m \ZZ) \to \GSp_{2g}(\ZZ/2^n \ZZ)$ and $\fy{2^m}{2^n} \colon \Sp_{2g}(\ZZ/2^m \ZZ) \to \Sp_{2g}(\ZZ/2^n \ZZ)$ be the natural projection maps. (When $m = \infty$, $\ZZ/2^m\ZZ$ denotes $\ZZ_2$.)
\item Let $H(2^k) = \gify{2^\infty}{2^k}(H) \subset \GSp_{2g}(\ZZ/2^k \ZZ)$ be the mod-$2^k$ reduction of $H$.
\item For any topological group $G$, let $[G,G]$ be the closure of its commutator subgroup, and let $G^{\ab} \defeq G/{ {[G,G]}}$ be its abelianization.
\item For each positive integer $n$, let $\id_n$ denote the $n \times n$ identity matrix.
\end{itemize}

\subsubsection{Main Group Theoretic Result}

We can now state the main theorem of this section.

\begin{theorem} \label{theorem:small-ab}
	Let $g \geq 2$. Let $H \subset \GSp_{2g}(\bz_2)$ be a subgroup such that $H = \gify{2^\infty}{2}^{-1}(H(2))$    and such that $H(2)$ contains $S_{2g+1}$.
    Then we have that
	\begin{equation}\label{thefirstclaim}
    [H, H] = \fy{2^\infty}{2}^{-1}([H(2), H(2)]).
    \end{equation}
	Moreover, the homomorphism \(H \to (H(2))^{\ab} \times (\bz_2)^\times\), defined on the left component by postcomposing reduction mod-2 with the abelianization map $H(2) \to H(2)^{\on{ab}}$ and on the right component by the multiplier map $\on{mult}$, induces an isomorphism
    \begin{equation}\label{thesecondclaim}
    H^{\on{ab}} \simeq (H(2))^{\ab} \times (\bz_2)^\times.
    \end{equation}
\end{theorem}

The relevance of Theorem~\ref{theorem:small-ab} to studying Galois representations of Jacobians of hyperelliptic curves is described in Lemma~\ref{theorem:r=2}, given at the beginning of Section~\ref{section:proof-of-mainbldg}.
We prove Theorem~\ref{theorem:small-ab} next in Section~\ref{subsection:proof-of-first-claim}.

\subsection{Proof of Theorem~\ref{theorem:small-ab}}
\label{subsection:proof-of-first-claim}
\begin{proof}[Proof of Theorem~\ref{theorem:small-ab} assuming
	Corollary~\ref{lemma:contains-mod-4} and
Proposition~\ref{lemma:contains-mod-2}]
Because we have that $$[H, H](2) = [H(2), H(2)],$$ in order to prove~\eqref{thefirstclaim}, it suffices to prove that
\(
	[H, H] \supset \ker \fy{2^\infty}{2}.
\)
To prove this statement, it further suffices to prove the following two statements:
\begin{enumerate}
\item[\customlabel{property-a}{(A)}] \(
	[H, H] \supset \ker \fy{2^\infty}{4},
\)
\item[\customlabel{property-b}{(B)}] \(
	[H(4), H(4)] \supset \ker \fy{4}{2}.
\)
\end{enumerate}
Statement~\ref{property-a} is proven in
Corollary~\ref{lemma:contains-mod-4}
and statement~\ref{property-b} is proven in
Proposition~\ref{lemma:contains-mod-2}.
To complete the proof, we only need verify~\eqref{thesecondclaim}. Note that~\eqref{thefirstclaim} tells us that the map $H \to (H(2))^{\ab} \times (\bz_2)^\times$ has kernel precisely $[H,H]$, so to prove~\eqref{thesecondclaim}, it suffices to check that the map $H \to (H(2))^{\ab} \times (\bz_2)^\times$ is surjective.
But this is easy to check by hand:
For $\alpha \in (\mathbb Z_2)^\times$, let $N_\alpha$ be the matrix which
has alternating $1$'s and $\alpha$'s on the diagonal, taken with respect to
a symplectic basis $e_1, \ldots, e_{2g}$ where $\langle e_{i}, e_{j} \rangle$
is $1$ if $i = 2k, j = 2k+1$ for some integer $k$, is $-1$ if $i = 2k+1, j = 2k$, and is $0$ otherwise.
For $(M_2, \alpha) \in (H(2))^{\ab} \times (\bz_2)^\times$, let $M_2^\infty \in \fy{2^\infty}{2}^{-1}(M_2)$, and observe that $M_2^\infty \cdot N_\alpha \mapsto (M_2, \alpha)$ via the map $H \to (H(2))^{\ab} \times (\bz_2)^\times$. This concludes the proof of our main group-theoretic result, Theorem~\ref{theorem:small-ab}.
\end{proof}
\subsubsection{Proving Statement~\ref{property-a}}\label{honeyseek}

We begin with the following lemma, in which we compute the commutator subalgebra of $\mf{gsp}_{2g}(\ZZ / 2 \ZZ)$.

\begin{lemma} \label{lemma:gsp-perfect}
	Let $\ell$ be a prime number. We have
	\(
		[\mf{gsp}_{2g}(\bz / \ell \bz), \mf{gsp}_{2g}(\bz / \ell \bz)] = \mf{sp}_{2g}(\bz / \ell \bz).
	\)\footnote{This result is a variant of~\cite[Proposition 2.10]{landesman-swaminathan-tao-xu:rational-families}.}
\end{lemma}

\begin{proof}
For convenience, let $\mf{g}_\ell \defeq [\mf{gsp}_{2g}(\bz / \ell \bz), \mf{gsp}_{2g}(\bz / \ell \bz)]$. That $\mf{g}_\ell \subset \mf{sp}_{2g}(\ZZ/\ell\ZZ)$ is obvious from the definitions of $\mf{gsp}_{2g}(\ZZ/\ell\ZZ)$ and $\mf{sp}_{2g}(\ZZ/\ell\ZZ)$, so it suffices to prove that reverse containment. For $\ell \geq 3$, this is immediate from~\cite[Proposition 2.10]{landesman-swaminathan-tao-xu:rational-families}, so we may restrict to the case where $\ell = 2$ (note that this is the case of primary interest to us).\footnote{In essence, the reason why the case of $\ell$ odd needs to be handled separately is that $\mf{sp}_{2g}(\ZZ/\ell\ZZ)$ is a perfect Lie algebra if and only if $\ell$ is odd, a result due to Hogeweij~\cite{hog1982}.} Choose a basis for $(\ZZ/2\ZZ)^{2g}$ with respect to which $\Omega_{2g}$ is given by
    \(
    	\Omega_{2g} = \left[ \begin{array}{c|c} 0 & \id_g \\  \hline  -\id_g & 0 \end{array} \right].
    \)
    Then $\mf{sp}_{2g}(\bz / 2 \bz)$ consists of matrices of the form
    \(
   	\left[ \begin{array}{c|c} A & B \\ \hline  C & -A^T \end{array} \right]
    \)
   where $A,B,C \in \on{Mat}_{g \times g}(\ZZ/2\ZZ)$ and $B,C$ are required to be symmetric. Since we have
\begin{align}
	\label{equation:block-diagonal-commutator}
 \left[ \left[\begin{array}{c|c} A & 0 \\ \hline 0 & -A^T \end{array}\right], \left[\begin{array}{c|c} D & 0 \\ \hline 0 & -D^T \end{array}\right] \right] & = \left[\begin{array}{c|c} AD - DA & 0 \\ \hline 0 & A^TD^T - D^TA^T \end{array}\right],
 \intertext{all block-diagonal matrices in $\mf{sp}_{2g}(\ZZ/2\ZZ)$ with every diagonal entry equal to $0$ are contained in $\mf{g}_2$. Moreover, for symmetric $B,C,E,F \in \on{Mat}_{g \times g}(\ZZ/2\ZZ)$, we have}
\label{equation:block-off-diagonal-commutator}
 \left[ \left[\begin{array}{c|c} 0 & B \\ \hline C & 0 \end{array}\right], \left[\begin{array}{c|c} 0 & E \\ \hline F & 0 \end{array}\right] \right] & = \left[\begin{array}{c|c} BF - EC & 0 \\ \hline 0 & CE - FB \end{array}\right],
 \intertext{and we can arrange that $BF-EC$ is an elementary matrix with a single nonzero entry on the diagonal.
	 Summing matrices from~\eqref{equation:block-diagonal-commutator} and~\eqref{equation:block-off-diagonal-commutator},
tells us that all block-diagonal matrices in $\mf{sp}_{2g}(\ZZ/2\ZZ)$ are contained in $\mf{g}_2$. Additionally, note that $\mf{gsp}_{2g}(\bz / 2 \bz)$ also contains
   	\(\left[ \begin{array}{c|c} \id_g & 0 \\ \hline  0 & 0  \end{array} \right], \)
    from which we deduce that $\mf{g}_2$ contains }
	    	\left[ \left[ \begin{array}{c|c} \id_g & 0 \\ \hline  0 & 0  \end{array} \right], \left[ \begin{array}{c|c} 0 & B \\ \hline  0 & 0 \end{array} \right] \right] &= \left[ \begin{array}{c|c} 0 & B \\ \hline  0 & 0 \end{array} \right],
\end{align}
where $B \in \on{Mat}_{g \times g}(\ZZ/2\ZZ)$ is symmetric. One similarly checks that $\mf{g}_2$ contains
\(\left[ \begin{array}{c|c} 0 & 0 \\ \hline  C & 0 \end{array} \right],\)
for $C \in \on{Mat}_{g \times g}(\ZZ/2\ZZ)$ symmetric. It follows that $\mf{g}_2 \supset \mf{sp}_{2g}(\ZZ/2\ZZ)$. \qedhere
\end{proof}

\begin{corollary} \label{corollary:coolwhip}
	We have
	\(
		[\ker \gify{2^\infty}{2}, \ker \gify{2^\infty}{2}] = \ker \fy{2^\infty}{4}.
	\)
\end{corollary}
\begin{proof}
Clearly \([\ker \gify{2^\infty}{2}, \ker \gify{2^\infty}{2}] \subset \ker \fy{2^\infty}{4}\), so it suffices to prove the reverse inclusion. By~\cite[Lemma 2.11]{landesman-swaminathan-tao-xu:rational-families}, we have that $[\ker \gify{2^\infty}{2}, \ker \gify{2^\infty}{2}] \supset \ker \fy{2^\infty}{8}$. Then, identifying $\mf{gsp}_{2g}(\ZZ/ 2 \ZZ)$ with $\ker \gify{8}{4}$ and $\mf{sp}_{2g}(\ZZ/2\ZZ)$ with $\ker \fy{8}{4}$, we have by Lemma~\ref{lemma:gsp-perfect} that
\begin{align*}
\ker \fy{8}{4} & = \left[ \ker \gify{8}{4},\ker \gify{8}{4} \right] = \left[ \ker \gify{2^\infty}{4},\ker \gify{2^\infty}{4} \right](8) \\
& \subset \left[ \ker \gify{2^\infty}{2},\ker \gify{2^\infty}{2} \right](8)
\end{align*}
It follows that $\left[ \ker \gify{2^\infty}{2},\ker \gify{2^\infty}{2} \right] \supset \ker \fy{2^\infty}{4}$.
\end{proof}

\begin{corollary} \label{lemma:contains-mod-4}
	We have $\ker \fy{2^\infty}{4} \subset [H, H]$.
\end{corollary}
\begin{proof}
	The hypothesis that
    \(
    	H = \gify{2^\infty}{2}^{-1}(H(2))
    \)
    implies that $\ker \gify{2^\infty}{2} \subset H$, and hence $[\ker \gify{2^\infty}{2}, \ker \gify{2^\infty}{2}] \subset [H, H]$. Applying Corollary~\ref{corollary:coolwhip} then yields the desired result.
\end{proof}

\subsubsection{Proving Statement (B): Tensor Product Notation}\label{tomswifties}

Just as we did in the proof of Lemma~\ref{lemma:gsp-perfect}, we must choose a basis with respect to which our symplectic form has an easy-to-use matrix representation. The goal of this subsection is to choose such a basis and to develop a shorthand notation for this basis. In Sections~\ref{subsection:s-action} and~\ref{subsection:finishing-the-proof}, we use this notation to prove Statement (B), thereby completing the proof of Theorem~\ref{theorem:small-ab}.

Recall the notation introduced in the first paragraph of Section~\ref{subsection:stimpy}: $R$ is a commutative ring (which we will take to be either $\ZZ_2$ or $\ZZ/4\ZZ$), and $M$ is a free $R$-module of rank $2g$. We choose a basis $(e_1, \ldots, e_{2g})$ for $M$ so that the symplectic form $\langle -, -\rangle$ is given by
$$\langle e_i , e_j \rangle \defeq \begin{cases} j-i & \text{ if $|j-i| = 1$ and $\max\{i,j\} \equiv 0\,(\bmod\,2$) } \\ 0 & \text{ otherwise} \end{cases}$$

We may alternatively construct $M$ as follows. Let $N_1 \simeq R^{2}$ have basis $(x_1, x_2)$ and let $N_2 \simeq R^{g}$ have basis $(y_1, \ldots, y_g)$. Endow $N_1$ with the alternating form given by $\langle x_i, x_j \rangle = j - i$, and endow $N_2$ with the symmetric form given by $\langle y_i, y_j \rangle = \delta_{ij}$, where $\delta_{ij}$ denotes the Kronecker $\delta$-function as usual. Then if we take $M \defeq N_1 \otimes N_2$, we have that $(x_i \otimes y_j : i \in \{1,2\}, \, j \in \{1, \dots, g\} )$ is a basis for $M$ and that $M$ is equipped with an alternating form defined on simple tensors by
\[
\langle a_1 \otimes b_1, a_2 \otimes b_2 \rangle \defeq \langle a_1, a_2 \rangle \cdot \langle b_1, b_2 \rangle.
\]
Note that the map sending $x_i \otimes y_j \mapsto e_{2j + i - 2}$ gives an identification between our two different constructions of $M$.

Linear operators on $M$ are $R$-linear combinations of tensor products of linear operators on $N_1$ with linear operators on $N_2$. If we denote by $x_{ij}$ the row-$i$, column-$j$ elementary matrix acting on $N_1$ and by $y_{mn}$ the row-$m$, column-$n$ elementary matrix acting on $N_2$, then a basis for $\on{End}(M)$ is given by $(x_{ij} \otimes y_{mn} : i,j \in \{1, 2\},\, m,n \in \{1, \dots, g\})$. Also notice that any element $\Lambda \in \on{End}(M)$ may be expressed as
\begin{equation}\label{fouseytattoo}
\Lambda =  \sum_{i=1}^g \sum_{j=1}^g \Lambda_{ij} \otimes y_{ij}
\end{equation}
where $\Lambda_{ij} \in \on{End}(N_1)$ for all $i,j \in \{1,\dots, g\}$.
\begin{proposition} \label{proposition:lie}
	Let $\phi \in \on{End}(\on{End}(N_1))$ be defined by
	\[
	x_{11} \mapsto -x_{22}, \quad x_{22} \mapsto -x_{11}, \quad x_{12} \mapsto x_{12}, \quad x_{21} \mapsto x_{21}.
	\]
The Lie algebra $\mf{gsp}_{2g}(R)$ consists of those elements $\Lambda \in \on{End}(M)$ with $\Lambda_{ij} \in \on{End}(N_1)$ such that there exists $d \in R$ satisfying
    $$\phi(\Lambda_{ji}) = \Lambda_{ij} - (d\delta_{ij}) \cdot \id_2.$$
Moreover, $\mf{sp}_{2g}(R)$ admits an analogous description in which $d$ is required to be zero.
\end{proposition}
\begin{proof}
Since $\Omega_{2g}^2 = -\id_{2g}$, the defining equation for $\mf{gsp}_{2g}(R)$ is equivalent to
\begin{align}	
\Omega_{2g} \Lambda^T \Omega_{2g} - \Lambda + d \cdot (\id_2 \otimes \id_g) & = 0. \label{equation:lie-algebra-condition}
\end{align}
Note that the identity element $\on{End}(N_2)$ is given by $\id_g = y_{11} + \cdots + y_{gg}$. Substituting this in along with the expansion~\eqref{fouseytattoo} for $\Lambda$ as well as $(x_{12} - x_{21}) \otimes \id_g$ for $\Omega_{2g}$ on the left-hand side of~\eqref{equation:lie-algebra-condition} yields that
\[
\sum_{i =1}^g \sum_{j=1}^g \big[ (x_{12} - x_{21})(\Lambda_{ji})^T(x_{12} - x_{21}) - \Lambda_{ij} + (d\delta_{ij}) \cdot \id_2 \big] \otimes y_{ij} = 0,
\]
which is equivalent to the following condition:
\[
(x_{12} - x_{21})(\Lambda_{ji})^T(x_{12} - x_{21}) = \Lambda_{ij} - (d\delta_{ij}) \cdot \id_2
\]
The desired result then follows upon observing that
	\(
	(x_{12} - x_{21})(\Lambda_{ji})^T(x_{12} - x_{21}) = \phi(\Lambda_{ji}). \qedhere
	\)
\end{proof}

\begin{remark} \label{remark:sp-basis}
	When $R = \ZZ/2\ZZ$, minus signs may be ignored, so the operator $\phi$ may be concisely described as transposition across the anti-diagonal. It follows from Proposition~\ref{proposition:lie} that the following is a basis for $\mf{sp}_{2g}(\ZZ/2\ZZ)$:
	\begin{align*}
    & \big(\id_2 \otimes y_{ii}, x_{12} \otimes y_{ii}, x_{21} \otimes y_{ii} : i \in \{1, \dots, g\}\big)\,\, \cup \\
    & \big(x_{12} \otimes (y_{ij}+y_{ij}) , x_{11} \otimes y_{ij} + x_{22} \otimes y_{ji} ,  x_{21} \otimes (y_{ij} + y_{ji}) , x_{22} \otimes y_{ij} + x_{11} \otimes y_{ji} : 1 \leq i < j \leq g \big).
\end{align*}
In Section~\ref{subsection:finishing-the-proof}, it will be convenient to define a function $\on{ind}$ that assigns to each of the above basis elements the value of $i$ (e.g., $\on{ind}(\id_2 \otimes y_{ii}) = i$ and $\on{ind}(x_{12} \otimes (y_{ij} + y_{ij})) = i$).
\end{remark}

\subsubsection{Proving Statement (B): Describing the Action of $S_{2g+2}$} \label{subsection:s-action}

We now seek to describe the embedding $S_{2g+2} \hookrightarrow \Sp_{2g}(\bz / 2 \bz)$ from Lemma~\ref{lemma:include-s} in terms of the tensor product notation that we just introduced in Section~\ref{tomswifties}. To this end, we set $R = \ZZ/2\ZZ$, so that $M \simeq \FF_2^{2g}$.
\begin{lemma} \label{lemma:basis}
	Recall notation from the proof of Lemma~\ref{lemma:include-s}. The map $\psi: M \ra t^\perp / \langle t \rangle$ of symplectic vector spaces defined by
	\[ x_1 \otimes y_n \mapsto \sum_{i=1}^{2n} e_i \quad \text{and} \quad
	x_2 \otimes y_n \mapsto e_{2n+1} + \sum_{i=1}^{2n-1} e_i \quad \text{for each} \quad n \in \{1, \dots, g\}
	\]
    is an isomorphism.
\end{lemma}
\begin{proof}
The lemma follows immediately from the observation that $\psi$ identifies the symplectic forms of $M$ and $t^\perp / \langle t \rangle$.
\end{proof}
Recall that the group $S_{2g+2}$ is generated by the adjacent transpositions $T_k$ for $k \in \{1, \dots, 2g+1\}$ whose cycle types are given by $T_k = (k, k+1)$. We now compute the action of $T_k$ on $M$ for each $k$:
\begin{lemma} \label{lemma:s-action}
	When viewed as operators on $M$, the transpositions $T_k$ are given by
	\begin{align*}
	T_{2n} &= \id_{2g} + (x_{11} + x_{12} + x_{21} + x_{22}) \otimes y_{nn}, \\
	T_{2n+1} &= \id_{2g} + x_{12} \otimes (y_{nn} + y_{(n+1) n} + y_{n (n+1)} + y_{(n+1) (n+1)}),
	\end{align*}
	according as $k = 2n$ or $k = 2n+1$, where any term with an out-of-range index is zero.
\end{lemma}
\begin{proof} 	
	The result for $k = 2n$ follows from the observation that $T_{2n}$ swaps $x_1 \otimes y_n$ with $x_2 \otimes y_n$ and keeps all the other basis vectors fixed. As for $k = 2n+1$, we break into three cases:
	\begin{enumerate}
		\item Suppose $n = 0$. The transposition $T_1$ sends
		$$\psi(x_2 \otimes y_1) = ( e_1 + e_3)  \mapsto (e_1 + e_2) + (e_1 + e_3)
		= \psi(x_1 \otimes y_1) + \psi(x_2 \otimes y_2)$$
		and fixes all other $\psi(x_i \otimes y_j)$. Thus, $T_1$ is given by
		\(
		T_1 = \id_{2g} + x_{12} \otimes y_{11}.
		\)
		\item Suppose $n = g$. The transposition $T_{2g+1}$ sends
		$$ \psi(x_2 \otimes y_g) = e_{2g + 1} + \sum_{i=1}^{2g-1} e_i  \mapsto  \sum_{i=1}^{2g+1} e_i + \sum_{i=1}^{2g-1} e_i  = e_{2g} + e_{2g+1} = \psi(x_1 \otimes y_g) + \psi(x_2 \otimes y_g) $$
		and fixes all other $\psi(x_i \otimes y_j)$. Thus, $T_g$ is given by
		\(
		T_{2g+1} = \id_{2g} + x_{12} \otimes y_{gg}.
		\)
		\item Finally, suppose $n \in \{1, \dots, g-1\}$. The transposition $T_{2n+1}$ sends
		\begin{align*}
		\phi(x_2 \otimes y_n) = e_{2n + 1} + \sum_{i=1}^{2n-1} e_i \,\,\, & \mapsto \,\,\, \sum_{i=1}^{2n+2} e_i + \left( e_{2n+1} + \sum_{i=1}^{2n-1} e_i \right) + \sum_{i=1}^{2n} e_i \\
		&\hphantom{\quad\mapsto}= \psi(x_1 \otimes y_{n+1}) + \psi(x_2 \otimes y_n) + \psi(x_1 \otimes y_n), \\
		\psi(x_2 \otimes y_{n+1})
		=  e_{2n + 3} + \sum_{i=1}^{2n+1} e_i \,\,\, &\mapsto \,\,\, e_{2n+3} + \sum_{i=1}^{2n+1} e_i + \sum_{i=1}^{2n+2} e_i + \sum_{i=1}^{2n} e_i \\
		&\hphantom{\quad\mapsto}= \psi(x_2 \otimes y_{n+1}) + \psi(x_1 \otimes y_{n+1}) + \psi(x_1 \otimes y_n),
		\end{align*}
		and fixes all other $\psi(x_i \otimes y_n)$. Thus, $T_{2n+1}$ is given by
		\[
		T_{2n+1} = \id_{2g} + x_{12} \otimes (y_{nn} + y_{(n+1) n} + y_{n (n+1)} + y_{(n+1) (n+1)}).
		\]
	\end{enumerate}
	The result for $k = 2n+1$ follows immediately from points (1)--(3) above.
\end{proof}

\subsubsection{Finishing the Proof of Statement (B)} 	
\label{subsection:finishing-the-proof}
\begin{proposition} \label{lemma:contains-mod-2}
	We have $[H(4), H(4)] \supset \ker \fy{4}{2}$.
\end{proposition}
\begin{proof}
	The assumption that
    \(
    	H = \gify{2^\infty}{2}^{-1}(H(2))
    \)
    implies that
	\(
	\ker \gify{4}{2} \subset H(4).
	\)
	Recall that we may identify $\ker \gify{4}{2}$ with $\mf{gsp}_{2g}(\ZZ/2\ZZ)$, so that each $S \in \ker \gify{4}{2}$ may be expressed as $S = \id_{2g} + 2\Lambda$ where $\Lambda \in \mf{gsp}_{2g}(\bz / 2 \ZZ)$. The assumption that $H(2)$ contains $S_{2g+1}$ tells us that for any $M_2 \in S_{2g+1} \subset \Sp_{2g}(\bz / 2 \bz)$, we may lift $M_2$ to an element $M_4 \in H(4)$. In particular, we have that
$$\id_{2g} + 2(\Lambda + M_2 \Lambda M_2^{-1}) = (\id_{2g} + 2\Lambda)^{-1}M_4(\id_{2g} + 2\Lambda)M_4^{-1} \in [H(4), H(4)].$$
To complete the proof, it suffices to show that matrices of the form $\Lambda + M_2 \Lambda M_2^{-1}$ span all of $\mf{sp}_{2g}(\bz / 2\bz)$. Let $V = \on{span}(\Lambda + M_2 \Lambda M_2^{-1} : \Lambda \in \mf{gsp}_{2g}(\ZZ/2\ZZ) \text{ and } M_2 \in S_{2g+1})$.
	
	It suffices to restrict our consideration to matrices $M_2$ corresponding to transpositions $T_k \in S_{2g+1}$. Note that $T_k = (T_k)^{-1}$, so that if we write $T_k = \id_{2g} + N_k$, then we have
    \begin{align} \label{equation:nuns-on-the-run}		
	\Lambda + T_k\Lambda(T_k)^{-1} = N_k\Lambda + \Lambda N_k + N_k \Lambda N_k.
	\end{align}	
	As in Lemma~\ref{lemma:s-action}, we will have to treat the cases $k = 2n$ and $k = 2n+1$ separately. In what follows, we induct on the value of the $\on{ind}$ function that $V$ contains the seven types of basis elements listed in Remark~\ref{remark:sp-basis}.
   First, however, we perform some calculations that serve to greatly simplify this inductive argument. Combining Lemma~\ref{lemma:s-action} with~\eqref{equation:nuns-on-the-run} and taking $\Lambda = x_{11} \otimes \id_g$, we find that
\begin{align}
\Lambda + T_{2n} \Lambda T_{2n} & = \id_2 \otimes y_{nn} \in V, \label{jc1}\\
\Lambda + T_{2n-1} \Lambda T_{2n-1} & = x_{12} \otimes (y_{(n-1)(n-1)} + y_{(n-1)n} + y_{n(n-1)} + y_{nn}) \in V. \label{jc2} \\
\intertext{Repeating this calculation for $k = 2n$ but taking $\Lambda = x_{12} \otimes y_{nn}$, we find that}
	\Lambda + T_{2n} \Lambda T_{2n} & = (x_{12} + x_{21}) \otimes y_{nn} \in V. \label{jc9} \\
\intertext{Now fix $n, \ell$ with $\ell > n$, and take $\Lambda = M \otimes y_{n\ell} + \phi(M) \otimes y_{\ell n}$ for any $M \in \on{Mat}_{2 \times 2}(\bz / 2 \bz)$. By Proposition~\ref{proposition:lie}, all such $\Lambda$ are elements of $\mf{gsp}_{2g}(\bz/2\bz)$. We find that}
\Lambda + T_{2n} \Lambda T_{2n} & = (x_{11} + x_{12} + x_{21} + x_{22})M \otimes y_{n\ell} \, +  \label{jc5}\\
& \hphantom{===} \phi(M)(x_{11} + x_{12} + x_{21} + x_{22}) \otimes y_{\ell n} \in V, \nonumber \\
\Lambda + T_{2n-1}\Lambda T_{2n-1} & = x_{12}M \otimes (y_{n \ell} + y_{(n-1)\ell}) + \phi(M) x_{12} \otimes (y_{\ell n} + y_{\ell (n-1)}) \in V \label{jc3}. \\
\intertext{Taking $M = x_{11}$ in~\eqref{jc5}, so that $\phi(M) = x_{22}$, yields that}
	& (x_{11} + x_{21}) \otimes y_{n \ell} + (x_{21} + x_{22}) \otimes y_{\ell n} \in V, \label{jc6} \\
    \intertext{and taking $M = x_{22}$ in~\eqref{jc5}, so that $\phi(M) = x_{11}$, yields that}
	& (x_{12} + x_{22}) \otimes y_{n\ell} + (x_{11} + x_{12}) \otimes y_{\ell n} \in V. \label{jc7} \\
\intertext{Taking $M = x_{22}$ in~\eqref{jc3}, so that $\phi(M) = x_{11}$, yields that}
& x_{12} \otimes (y_{n \ell} + y_{\ell n}) \in V \label{jc4} \\
\intertext{and taking $M = x_{21}$ in~\eqref{jc4}, so that $\phi(M) = x_{21}$, yields that}
&  x_{11} \otimes y_{n \ell} + x_{22} \otimes y_{\ell n} \in V. \label{jc8}
\end{align}
We are now ready to carry out the induction. For the base case, we need to check that all basis vectors with $\on{ind}$-value equal to $1$ are in $V$; this follows immediately upon taking $n = 1$ in~\eqref{jc1}--\eqref{jc8}. Next, suppose for some $N \in \{1, \dots, g\}$ we have that all basis vectors with $\on{ind}$-value less than $N$ are in $V$. Taking $n = N$ in~\eqref{jc1}--\eqref{jc8} and applying the inductive hypothesis yields that all basis vectors with $\on{ind}$-value equal to $N$ are in $V$.
\end{proof}

\section{Proof of Theorem~\ref{mainbldg}}
\label{section:proof-of-mainbldg}

In this section, we prove the first main result of this paper, namely Theorem~\ref{mainbldg}. We begin in Section~\ref{mygawdsadiomaneisamazing} with a description of the relevant background material on Galois representations of PPAVs. Then, in Section~\ref{prelimswine},
we prove a group-theoretic Lemma, useful for determining $\delta_K$.
In Section~\ref{symbed}, we describe the particular manner in which we embed $S_{2g+2}$ as a subgroup of $\Sp_{2g}(\ZZ/2 \ZZ)$.
In Section~\ref{subsection:monodromy-of-families}
we determine the monodromy groups of the four families of hyperelliptic curves introduced in Definition~\ref{definition:standard-families} and the monodromy of the universal family over the moduli stack of hyperelliptic curves. Finally, in Section~\ref{32isnotagoodscoreline}, we complete the proof of Theorem~\ref{mainbldg}.

\subsection{Background}\label{mygawdsadiomaneisamazing}

Let $K$ be a number field, let $r \geq 0$ be an integer, and let $U \subset \mathbb{P}_K^r$ be an open subscheme. For an integer $g \geq 0$, let $A$ be a family of $g$-dimensional PPAVs over $U$, by which we mean that $A$ is an abelian scheme over $U$, meaning that $A \rightarrow U$ is a proper smooth group scheme with geometrically connected fibers of dimension $g$, and $A$ is equipped with a principal polarization over $U$.  Because the base $U$ is rational, we call $A \to U$ a \emph{rational family}. By construction, the fiber $A_u$ of the morphism $A \to U$ over any $K$-valued point $u \in U(K)$ is a $g$-dimensional PPAV over $K$.

Recall that the action of the \'{e}tale fundamental group $\pi_1(U)$ on the torsion points of a chosen geometric generic fiber of $A \to U$ gives rise to a continuous linear representation whose image is constrained by the Weil pairing to lie in the general symplectic group $\GSp_{2g}(\wh{\ZZ})$. We denote the resulting \emph{adelic representation} by
\begin{equation}\label{thisisthepartofme}
	\rho_A \colon \pi_1(U) \to \GSp_{2g}(\wh{\ZZ}).\footnote{The map in~\eqref{thisisthepartofme} is well-defined up to the choice of base-point, and choosing a different base-point would only alter the image of $\rho_{A}$ by an inner automorphism.
For this reason, when it will not lead to confusion, we may omit
the base-point from our notation.}
\end{equation}

We now define the monodromy groups associated to $\rho_A$. We call the image of $\rho_A \colon \pi_1(U) \to \GSp_{2g}(\wh{\ZZ})$ the {\it monodromy} of the family $A \to U$, and we denote it by $\mono_A$. We write $\mono_A(m)$ for the mod-$m$ reductions and $\mono_{A,\ell}$ for the $\ell$-adic reductions of the above-defined monodromy groups.

\begin{remark}
Let $u \in U(K)$ be a $K$-valued point. Precomposing the adelic representation with the induced map $\pi_1(u) \to \pi_1(U)$ gives a representation $\pi_1(u) \to \GSp_{2g}(\wh{\ZZ})$ whose image we denote by $\mono_{A_u}$.
Because $\pi_1(u) \simeq G_K$,
the representation $\rho_{A_u}$ obtained by restricting $\rho_A$ to $A_u$
is the same as the adelic representation $\rho_{A_u}$ discussed in Section~\ref{subsection:intro-background}.
\end{remark}

\begin{remark}
	\label{remark:det-rho-is-chi}
For a commutative ring $R$, recall from the definition of the general symplectic group that we have a multiplier map $\on{mult} \colon \GSp_{2g}(R) \to R^\times$. If $\chi$ denotes the cyclotomic character, then for a PPAV $A$ it follows from $G_K$-invariance of the Weil pairing that $\chi = \mult \circ \rho_{A}$.
More generally, if $A \rightarrow U$ is a family
of PPAVs with $U$ normal and integral, and if $\phi$ denotes the map $\pi_1(U) \rightarrow \pi_1(\spec K)$ induced by the structure map $U \to \spec K$, then we have that
$\chi \circ \phi = \mult \circ \rho_A$
\end{remark}

\subsection{Computing $\delta_K$}
\label{prelimswine}

In this section, we prove Lemma~\ref{theorem:r=2},
which is used to compute the value of $\delta_K$ in the
proof of Theorem~\ref{mainbldg}, given in Section~\ref{subsection:proof-of-first-claim}.
In order to state Lemma~\ref{theorem:r=2},
we need the following definition, in which
we introduce notation used throughout the paper to denote various lifts of $S_{2g+i}$:

\begin{definition}\label{grpnotes}
	For $i \in \left\{ 1,2 \right\}$,
	we define
	\begin{equation*}
 \hyphat_{2g+i} \defeq (\Sp_{2g}(\zh) \to \Sp_{2g}(\bz / 2 \bz))^{-1}(S_{2g+i}) \qquad \text{and} \qquad \hypaddict_{2g+i} \defeq \fy{2^\infty}{2}^{-1}(S_{2g+i}).
\end{equation*}
\end{definition}
The next lemma applies Theorem~\ref{theorem:small-ab} to determine how large the commutator subgroup of $\widetilde{\mathrm{G} \mathcal {S}}_{2g+i, K},$ is as a subgroup of $\hyphat_{2g+i}$:
\begin{lemma} \label{theorem:r=2}
	Let $g \ge 2$, let $i \in \left\{ 1,2 \right\}$
	and let $H \subset \GSp_{2g}(\zh)$ be a closed subgroup. Suppose that
\begin{itemize}
\item $H_2 = \gify{2^\infty}{2}^{-1}(S_{2g+i})$, and
\item $H(\ell) \supset \Sp_{2g}(\bz / \ell \bz)$ for $\ell \ge 3$.
\end{itemize}
Then
\[
[H, H] = \fy{2^\infty}{2}^{-1}(A_{2g+i}) \times \prod_{\ell \ge 3} \Sp_{2g}(\bz_\ell),
\]
where $A_{2g+i}$ denotes the alternating group on $2g + i$ letters.
\end{lemma}
\begin{proof}
By Theorem~\ref{theorem:small-ab},
\begin{align*}
[H, H]_2 = [H_2, H_2] = \fy{2^\infty}{2}^{-1}(A_{2g+i}).
\end{align*}
Also, note that for $\ell \ge 3$,
\begin{align*}
[H, H](\ell) &= [H(\ell), H(\ell)] \\
&\supset [\Sp_{2g}(\bz / \ell \bz), \Sp_{2g} (\bz / \ell \bz)] \\
&= \Sp_{2g}(\bz / \ell \bz),
\end{align*}
the last equality following from \cite[3.3.6]{omeara1978symplectic}.
We now appeal to the fact that a closed subgroup of $\Sp_{2g}(\mathbb{Z}_\ell)$ mapping onto $\Sp_{2g}(\mathbb{Z}/\ell\mathbb{Z})$ must in fact be all of $\Sp_{2g}(\mathbb{Z}_\ell)$ for $g > 1$.
This fact was shown in~\cite[Theorem B]{weigel:on-the-profinite-completion-of-arithmetic-groups-of-split-type} (except for the case where $g = 3$ and $\ell = 2$),
as well as in
~\cite[Theorem 1.3]{vasiu2003surjectivity},
and then again in~\cite[Proposition 2.5]{landesman-swaminathan-tao-xu:rational-families}.
Applying this fact to $[H, H] \subset \Sp_{2g}(\zh)$ gives the result.
\end{proof}
\begin{corollary} \label{theorem:r=2:special}
For $H$ as in Lemma~\ref{theorem:r=2}, we have $[\hyphat_{2g+i} : [H, H]] = 2$. In particular, $[\hyphat_{2g+i} : [
\widetilde{\mathrm{G} \mathcal {S}}_{2g+i, K},
\widetilde{\mathrm{G} \mathcal {S}}_{2g+i, K},
]] = 2$.
\end{corollary}

\subsection{Embedding the Symmetric Group, Take 2} \label{symbed}

In Section~\ref{take1} we constructed the well-known embedding $S_{2g+2} \hookrightarrow \Sp_{2g}(\bz / 2 \bz)$. Beginning with
\begin{align*}
	V &\simeq \FF_2^{2g+2} \\
	t &= (1, \dots, 1) \in V \\
	W &= t^{\perp}/\on{span}(t),
\end{align*}
we observed that the action of $S_{2g+2}$ on the basis vectors of $V$ descends to a symplectic action on $W$. Our goal in this section is to relate this embedding with the mod-2 Galois representation attached to a family of hyperelliptic curves, by proving the following result:
\begin{theorem} \label{monodromy-in-sym}
	Given a family $\mc{C} \to \mc{U}$ of hyperelliptic curves (where $\mc{U}$ is any stack), and a geometric generic point $\ol{\eta} \hookrightarrow \mc{U}$, the monodromy group $\rho_{2}(\pi_1(\mc{U}, \ol{\eta})) \subset \Sp_{2g}(\bz / 2 \bz)$ is in fact contained in $S_{2g+2} \subset \Sp_{2g}(\bz / 2 \bz)$. As a subgroup of $S_{2g+2}$, the monodromy group is given by the action of $\pi_1(\mc{U})$ on the Weierstrass points of $\mc{C}_{\ol{\eta}}$.
\end{theorem}
This has an immediate consequence for our standard families:
\begin{corollary} \label{sym-W}
	For $i \in \{1, 2, 3, 4\}$, we have $\mono_{\standardFamily g i K} \subset \grouphypboth g {(i \bmod 2)} K$.
\end{corollary}
\begin{proof}
	We have $\on{mult}(\mono_{\standardFamily g i K}) = \chi(K)$ as subgroups of $\zh^\times$, by Remark~\ref{remark:det-rho-is-chi}. Therefore, it suffices to show that $\mono_{\standardFamily g i K}(2) \subset S_{2g+2-(i \bmod 2)}$. By Theorem~\ref{monodromy-in-sym}, we need only check that the monodromy action on the Weierstrass points of $\standardFamily g i K \to \standardTarget g i K$ is contained in $S_{2g+2-(i \bmod 2))}$. The nontrivial cases $i = 1, 3$ follow by observing that, when the defining equation is $y^2 = f(x)$ with $\deg f(x) = 2g+1$, one Weierstrass point always lies over infinity, hence is fixed under monodromy.
\end{proof}

We prove Theorem~\ref{monodromy-in-sym} in three steps:
\begin{enumerate}[(1)]
	\item In subsection~\ref{sec:single}, we prove the statement when $\mc{U}$ is $\spec \ol{k}$. The key points are that the constructions are functorial in $\mc{C}$ and that the isomorphism with $\on{Jac}_{C/U}[2]$ follows from standard facts about divisors on hyperelliptic curves.
	\item In subsection~\ref{part2}, we prove the statement when $\mc{U}$ is a scheme by explicitly constructing the algebraic space of Weierstrass points over $\mc{U}$, the corresponding group space $t^\perp / \on{span}(t)$ over $\mc{U}$, and the map from the latter to $\on{Jac}_{C/U}[2]$. Step (1) implies that this map is an isomorphism.
	\item In subsection~\ref{part3}, we interpret the (functorial) constructions of step (2) as giving rise to corresponding objects and maps over the moduli space $\hyperell_g$ of hyperelliptic curves, corresponding to the universal family $\hyperellSource_g \to \hyperell_g$.
\end{enumerate}

\subsubsection{A single hyperelliptic curve} \label{sec:single}
Let $k$ be an algebraically closed field of characteristic zero, let $C$ be a hyperelliptic curve over $k$, and let $J$ be the Jacobian of $C$. The
set of Weierstrass points $\{P_1, \ldots, P_{2g+2}\}$ of $C$ is uniquely determined because $g \ge 2$. With this setup, define $V$
to be the free vector space over $\FF_2$ spanned by $P_1, \ldots, P_{2g+2}$, so that
\begin{align*}
	t &= P_1 + \cdots + P_{2g+2} \\
	t^\perp &= \on{span}_{\mathbb{F}_2}(P_i - P_j : i, j \in \{1, \ldots, 2g+2\}).
\end{align*}
The map
\[
\begin{tikzcd}[row sep = 0.15cm]
\on{span}_{\bz}(P_1, \ldots, P_{2g+2}) \ar{r}{\phi} & \on{Pic}_C \\
\sum_i a_i \cdot P_i  \ar[mapsto]{r} & \mc{O}_C\left( \sum_i a_i \cdot P_i \right),
\end{tikzcd}
\]
is such that $\phi(\on{span}_{\bz}(P_i - P_j)) \subset \on{Pic}_C^0 \simeq J$. Furthermore, it can be checked that
\begin{itemize}
	\item The resulting map $\on{span}_{\bz}(P_i - P_j) \to J$ annihilates $2\cdot (P_i - P_j)$ and $t$. Hence it descends to a map $W \defeq t^\perp / \on{span}(t) \to J[2]$ of $\FF_2$ vector spaces.
	\item This latter map is an isomorphism.
\end{itemize}
The second bullet point implies that the action of $\on{Aut}(C)$ on $J[2]$, \emph{a priori} contained in $\Sp_{2g}(\bz / 2 \bz)$, is in fact contained in the subgroup $S_{2g+2} \subset \Sp_{2g}(\bz / 2 \bz)$ which is determined by the vector space isomorphism $J[2] \simeq W$. For details, see \cite[Proposition 1.2.1(a)]{yelton2015thesis}.

\subsubsection{Schematic families of hyperelliptic curves} \label{part2}
Let ${C} \to {U}$ be a family of hyperelliptic curves of genus $g$, where $U$ is a scheme. Because all constructions in \ref{sec:single} were functorial, they can be carried out in families. Let us indicate how this is done.
\begin{enumerate}[itemsep=0.25cm]
	\item Let ${P}$ be the fixed point locus of the hyperelliptic involution. Then we have a diagram
	\[
	\begin{tikzcd}[column sep = 2 cm]
	{P} \ar[hookrightarrow]{r}{\text{closed emb.}} \ar[swap]{rd}{\text{\'etale}} & {C} \ar{d} \\
	& {U}
	\end{tikzcd}
	\]
	For any geometric point $u \hookrightarrow {U}$, the fiber ${P}_u$ consists of the Weierstrass points of ${C}_u$.
	\item Let ${G}$ be the group algebraic space over ${U}$ which represents the sheaf associated to the following presheaf on $\on{Sch}_{/{U}}$, in the \'etale topology:
	\[
	T \mapsto \on{span}_\bz (\on{Hom}_{{U}}(T, {P})).
	\]
	Representability follows by taking an \'etale cover of ${U}$ which trivializes ${P}$. For $u \ra {U}$ a geometric point, the fiber ${G}_u$ equals $\on{span}_\bz({P}_u)$.
	
	There is a section $t \colon {U} \to {G}$ which is first defined on a sufficiently fine \'etale cover ${U}' \to {U}$ for which $U' \times_U P$ is a trivial $(2g+2)$-cover of $U'$, by ``adding all the Weierstrass points,'' i.e.\ summing the $(2g+2)$ basis elements of
	\[
		\on{span}_\bz(\Hom_U(U', P)) \simeq \on{span}_\bz(\Hom_{U'}(U', U' \times_U P)).
	\]
	The section on $U'$ can then be descended to $U$.
	
	We can also define a group subspace ${G}^0 \hookrightarrow {G}$ via the sub-presheaf given by requiring that the coefficients of the $\bz$-linear combination sum to zero.
	\item Define a map $\Phi: {G} \to \on{Pic}_{{C}/{U}}$ of group spaces over ${U}$ as follows: given $f \in {G}(T)$, we may find an \'etale cover $\sigma: T' \to T$ for which $\sigma^* f = \sum_i f_i$, for some $f_i \in \Hom_{{U}}(T', {P})$. Each $f_i$ gives a section of the pulled-back family ${C}_{T'} \to T'$, whose image determines a relative effective Cartier divisor $D_i$. A standard descent argument shows that $\sum_i D_i$ descends to a divisor $D$ on ${C}_T$, which does not depend on the chosen \'etale cover $\sigma$. We may therefore define $\Phi(f) \defeq D$. This assignment is natural in $T$, so it gives a natural transformation of functors ${G} \to \on{Pic}_{{C}/{U}}$. The fiber $\Phi_u$ is the map $\phi$ defined in \ref{sec:single}. The map $\Phi$ restricts to a map ${G}^0 \to \on{Pic}_{{C}/{U}}^0 \simeq \on{Jac}_{{C}/{U}}$.
	\item Subsection \ref{sec:single}
    allows us to describe the kernel and image of $\Phi$ as follows. First, $2 \cdot G^0$ maps to zero, so $\Phi$ descends to a map $G^0 / (2 \cdot G^0) \to \on{Jac}_{C/U}$. Second, the inclusion $G^0 \hookrightarrow G$ gives an injection $G^0 / (2 \cdot G^0) \hookrightarrow G / (2 \cdot G)$, and the image of $t \in G(U)$ in the quotient $G / (2\cdot G)$ in fact lies in $G^0 / (2 \cdot G^0)$ because $t$ is a sum of an even number of terms; we abuse notation by denoting the latter section with the same symbol $t$. This $t$ spans the kernel of the descended map $G^0 / (2 \cdot G^0) \to \on{Jac}_{C/U}$, the image of which is equal to $\on{Jac}_{C/U}[2]$. Thus, we have that
	\[
		G^0 / (2 \cdot G^0 + \on{span}(t)) \to \on{Jac}_{C/U}[2]
	\]
	is an isomorphism of group stacks over $U$.
\end{enumerate}
This proves Theorem~\ref{monodromy-in-sym} when $\mc{U} \defeq U$ is a scheme, because the action of $\pi_1(U)$ on the fiber of $G^0 / (2G^0 + \on{span}(t))$ over a chosen geometric generic point $\ol{\eta} \in U$ (as an $\mathbb{F}_2$-vector space) is obtained from the action of $\pi_1(U)$ on the fiber of $P_{\ol{\eta}}$ (as a set of size $(2g+2)$) via the procedure of Section~\ref{take1}.

\subsubsection{The universal family over $\hyperell_g$} \label{part3}

The constructions of subsection~\ref{part2} are functorial in the chosen family $C \to U$, and behave well under base change $V \to U$, so we obtain the analogous constructions over the moduli stack of hyperelliptic curves:
\[
\begin{tikzcd}[column sep = 0.7in]
\mc{P} \ar[hookrightarrow]{r}{\text{closed emb.}} \ar[swap]{rd}{\text{\'etale}} & \hyperellSource_g \ar{d} \\
& \hyperell_g
\end{tikzcd}
\quad \text{and} \quad
\begin{tikzcd}[column sep = 0.7in]
\mc{G}^0 \ar{r}{\mc{O}(-)} \ar[hookrightarrow]{d} & \on{Pic}^0_{\hyperellSource_g/\hyperell_g} \ar[hookrightarrow]{d} \\
\mc{G} \ar{r}{\mc{O}(-)} \ar{dr} & \on{Pic}_{\hyperellSource_g/\hyperell_g} \ar{d} \\
& \hyperell_g \ar[dashed, bend left = 20]{ul}{t}
\end{tikzcd}
\]
where $t$ is a section of $\mc{G} \to \hyperell_g$, for which we have the following isomorphism of group stacks over $\hyperell_g$:
\[
\begin{tikzcd}
\mc{G}^0/(2 \cdot \mc{G}^0 + \on{span}(t)) \ar{r}{\simeq} \ar{dr} & \on{Pic}^0_{\hyperellSource_g/\hyperell_g}[2] \ar{d} \\
& \hyperell_g
\end{tikzcd}
\]
Here, $t$ is interpreted as a section of $\mc{G}^0 / (2 \cdot \mc{G}^0)$ over $\hyperell_g$ just as in step (4) of subsection~\ref{part2}.

\begin{remark} By way of example, let us explain the definition of $\mc{P}$, and prove that it is a stack. By definition, $\mc{P}(U)$ is the groupoid whose objects are pairs $(C \to U, f)$ where $C \to U$ is a hyperelliptic family over $U$ (i.e.\ an object of $\hyperell_g(U)$) and $f$ is a section of the \'etale cover $P \to U$ constructed in subsection~\ref{part2} from the family $C \to U$. A morphism $(C_1 \to U, f_1) \simeq (C_2 \to U, f_2)$ is an isomorphism $\sigma$ of families
\[
\begin{tikzcd}
C_1 \ar{rr}{\simeq} \ar{dr} && C_2 \ar{dl} \\
& U
\end{tikzcd}
\]
for which the resulting isomorphism $P_1 \simeq P_2$ of the associated spaces of Weierstrass points identifies the sections $f_1$ and $f_2$.

The descent condition is easy to check. Given an \'etale cover $V \to U$, a descent datum for $\mc{P}$ is given by a family $\wt{C} \to V$ and a section $\wt{f}$ of the resulting space of Weierstrass points, denoted $\wt{P} \to V$, along with gluing isomorphisms that take place over $V \times_U V$, which satisfy a cocycle condition on $V \times_U V \times_U V$. The cocycle condition first allows us to realize $\wt{C}$ as the pullback of a family $C \to U$, because $\hyperell_g$ is known to be a stack. By functoriality, the pullback to $V$ of the resulting space of Weierstrass points $P \to U$ is canonically identified with $\wt{P} \to V$. So effectiveness of the descent datum follows from the fact that $P \to U$ is an \'etale sheaf over $\on{Sch}_{/U}$, and, as such, satisfies a gluing axiom.
\end{remark}
In a similar way, the following points are formal consequences of subsection~\ref{part2}:
\begin{itemize}
	\item All stacks appearing in the three commutative diagrams above are algebraic.
	\item All maps to $\hyperell_g$ appearing above are representable.
	\item The isomorphism in the third diagram, which gives the desired statement about monodromy for the universal family $\hyperellSource_g \to \hyperell_g$, can be checked on pullback to schemes $U$, but this is exactly the conclusion of subsection~\ref{part2}.
\end{itemize}

To finish the proof of Theorem~\ref{monodromy-in-sym}, we need only note that any hyperelliptic family $\mc{C} \to \mc{U},$ with $\mc U$ a Deligne-Mumford stack, is pulled back from the universal family $\hyperellSource_g \to \hyperell_g$ via a map $\mc{U} \to \hyperell_g$. In this case, all constructions above can be pulled back along the same map $\mc{U} \to \hyperell_g$. Therefore, to $\mc{C} \to \mc{U}$, we can associate a stack of Weierstrass points, whose $\bz$-span maps to $\on{Pic}_{\mc{C} / \mc{U}}$, giving rise to an isomorphism analogous to that of the third commutative diagram above. This gives the desired result, by the same reasoning as in the last paragraph of subsection~\ref{part2}.

\subsection{Monodromy of Hyperelliptic Families $\standardTarget g i K$ and $\hyperelliptic g K$}
\label{subsection:monodromy-of-families}

We now show that the containment in Corollary~\ref{sym-W} is an equality for the families $\standardFamily g i K \to \standardTarget g i K$.
\begin{lemma} \label{lemma:monodromy-std}
Let $g \geq 2$, and let $i \in \{1, 2, 3, 4\}$.
\begin{enumerate}[label=(\roman*)]
	\item For any algebraically closed field $L$ which is a subfield of $\mathbb C$, we have $\mono_{\standardFamily g i L}= \hyphat_{2g+2-(i \bmod 2)}$.
	\item For any number field $K$, we have $\mono_{\standardFamily g i K}= \grouphypboth g {(i \bmod 2)} {K}$.
\end{enumerate}
\end{lemma}
\begin{proof}
(i) $\Leftrightarrow$ (ii): we have a map of short exact sequences
\begin{equation}
	\label{equation:}
	\begin{tikzcd}
		0 \ar {r}  &  \mono_{\standardFamily g i {\overline K}} \ar {r}\ar{d} & \mono_{\standardFamily g i K} \ar {r}{\mult}\ar{d} & \chi(K) \ar {r}\ar[equal]{d} & 0 \\
			0 \ar {r} &  \hyphat_{2g+2-(i \bmod 2)} \ar {r} & \grouphypboth g {(i \bmod 2)} K \ar[swap]{r}{\mult} & \chi(K) \ar {r} & 0.
	\end{tikzcd}\end{equation}
By the Five Lemma, the second vertical map is an isomorphism if and only if the first is.

Proof of (i): By Corollary~\ref{sym-W}, and because $L$ is a subfield of $\bc$, we have containments
\[
H_{\standardFamily g i \bc} \subset H_{\standardFamily g i L} \subset \hyphat_{2g+2-(i \bmod 2)}.
\]
Therefore, it suffices to show that $H_{\standardFamily g i \bc} = \hyphat_{2g+2-(i \bmod 2)}$. For $i \in \{1, 2\}$, this follows from~\cite[Th\'eor\`eme 1]{acampo:tresses-monodromie-et-le-groupe-symplectique}, since the \'etale fundamental group is the profinite completion of the topological fundamental group. To complete the proof we need only show that, when $i \in \{3, 4\}$, we have $\mono_{\standardFamily g i {\mathbb C}} = \mono_{\standardFamily g {i-2} {\mathbb C}}$.

For this, it suffices to construct a deformation retract
\[
	\phi: \standardTarget g {i-2} {\mathbb C} \times [0, 1] \to \standardTarget g {i-2} {\mathbb C}
\]
of $\standardTarget g {i-2} {\mathbb C}$ onto $\standardTarget g i {\mathbb C}$, which is done as follows. Let $n \defeq 2g+2-(i \bmod 2)$. Then $\standardTarget g {i-2} {\mathbb C}$ parameterizes unordered $n$-tuples of distinct points in $\ba^1_\bc$, and $\standardTarget g i {\mathbb C}$ parameterizes those which sum to zero. At time $t \in [0, 1]$, we define
\[
\phi_t \colon \{z_i\}_{i=1}^n \mapsto \left\{z_i - t \cdot \frac{z_1 + \cdots + z_n}{n}\right\}_{i=1}^n,
\]
where the $n$-tuple on the right sums to zero by construction. This $\phi$ is continuous, as desired. In fact, $\phi$ is regular: its coordinate functions are obtained by expressing the elementary symmetric polynomials in the right hand side $n$-tuple as polynomials in (the elementary symmetric polynomials of the $z_i$) and $t$.
\end{proof}

\begin{corollary}
	\label{monodromy-stack}
	Let $g \geq 2$. We have that $\mono_{\hyperellipticSource g K} = \grouphyp g {K}$,
	where $\hyperellipticSource g K \ra \hyperelliptic g K$ is the universal family over the moduli stack of hyperelliptic curves.
\end{corollary}
\begin{proof}
\emph{A priori}, we have the containments
\[
\mono_{\standardFamily g 2 K} \subset \mono_{\hyperellipticSource g K} \subset \grouphyp g K,
\]
the latter following from Corollary~\ref{monodromy-in-sym}. But Lemma~\ref{lemma:monodromy-std} implies that $\mono_{\standardFamily g 2 K} = \grouphyp g K$.
\end{proof}

\subsection{Finishing the Proof}\label{32isnotagoodscoreline}

We are now in position to complete the proof of Theorem~\ref{mainbldg}.
The main input to the proof is the following general theorem from \cite{landesman-swaminathan-tao-xu:rational-families}.
\begin{theorem}[\protect{\cite[Theorem 1.1]{landesman-swaminathan-tao-xu:rational-families}}]

	\label{theorem:main}
	Let $B, n > 0$, and suppose that the rational family $A \to U$ is non-isotrivial and has big monodromy, meaning that $\mono_A$ is open in $\GSp_{2g}(\zh)$. Let $\delta_\QQ$ be the index of the closure of the commutator subgroup of $\mono_A$ in $\mono_A \cap \Sp_{2g}(\zh)$, and let $\delta_K = 1$ for $K \neq \QQ$. Then $[\mono_A : \mono_{A_u}] \geq \delta_K$ for all $u \in U(K)$, and we have the following asymptotic statements:
			\[
				\frac{|\{u \in U(K) \cap \mathcal{O}_K^r : \lVert u \rVert \le B,\, [\mono_A : \mono_{A_u}] = \delta_K\}|}{|\{u \in U(K) \cap \mc{O}_K^r : \lVert u \rVert \le B\}|} = 1 + O((\log B)^{-n}), \text{ and}
			\]
\[
				\frac{|\{u \in U(K) : \on{Ht}(u) \leq B,\, [\mono_A : \mono_{A_u}] = \delta_K\}|}{|\{u \in U(K) : \on{Ht}(u) \le B\}|} = 1 + O((\log B)^{-n}),
			\]
	 where the implied constants depend only on $A \to U$ and $n$.
\end{theorem}
\begin{proof}[Proof of Theorem~\ref{mainbldg}]
	First, we explain how Theorem~\ref{theorem:main} applies to the standard families $\standardFamily g i K \rightarrow \standardTarget g i K$. By Lemma~\ref{lemma:monodromy-std}(ii), these families have big monodromy, so Theorem~\ref{theorem:main} applies.  Lemma~\ref{lemma:monodromy-std}(ii) says that $H_{\standardFamily g i K} = \grouphypboth g {(i\bmod 2)} K$. With this in mind, Corollary~\ref{theorem:r=2:special} implies that $\delta_\bq = 2$ in the statement of Theorem~\ref{theorem:main}.

	Next, we apply Theorem~\ref{theorem:main} to a rational family $C \rightarrow U$ represented by a map $U \to \hyperelliptic g K$ with connected geometric generic fiber. This hypothesis implies that $\pi_1(U) \to \pi_1(\hyperelliptic g K)$ is a surjection, cf.\ \cite[Corollary 5.3]{landesman-swaminathan-tao-xu:rational-families}, so the monodromy group of $C \to U$ is equal to that of the universal family over $\hyperelliptic g K$, and Corollary~\ref{monodromy-stack} implies that the universal family over $\hyperelliptic g K$ has monodromy group $\grouphyp g K$. At this point, Corollary~\ref{theorem:r=2:special} implies $\delta_\bq = 2$, as before.
\end{proof}

\section{Verification of the Examples}\label{verified}

The objective of this section is to prove our second main result, namely Theorem~\ref{exemplinongratia}. To verify that the example curves stated in Theorem~\ref{exemplinongratia} have maximal monodromy among members of $\hyperelliptic{g}{\QQ}$, we shall rely on two different sets of criteria, one adapted from~\cite{anni2017constructing}, and the other adapted from~\cite{seaweed}. We introduce these criteria in Section~\ref{sweenytime}; then, in Section~\ref{icheckthat}, we apply these criteria to check the example curves.

\subsection{Criteria for Having Maximal Monodromy}\label{sweenytime}

Let $g \in \{2,3\}$, and let $C$ be a genus-$g$ hyperelliptic curve over $\QQ$ given by the Weierstrass equation $y^2 = f(x)$, where $f(x) \in \bq[x]$ is a polynomial of degree $2g+2$; note that $C$ is a $\mathbb Q$-valued point of $\standardTarget{g}{2}{\QQ}$.

Let $J$ denote the Jacobian of $C$. We want to write down criteria for the associated monodromy group $H_J$ to be as large as possible in $H_{\hyperelliptic{g}{\mathbb Q}} \simeq \grouphyp{g}{\QQ}$, which by Theorem~\ref{mainbldg} is equivalent to having index $2$ in $\grouphyp{g}{\QQ}$.
We shall rely on the following lemma, which gives us two conditions under which maximal monodromy is attained:
\begin{lemma}
	\label{lemma:reduction-from-adelic-to-mod-l}
	Suppose $C$ is a hyperelliptic curve over $\mathbb Q$ with Jacobian $J$ satisfying
	\begin{align}\label{thethingIwanttocheck}
(\mono_J)_2 &= \gify{2^\infty}{2}^{-1}(S_{2g+2}), \text{ and} \\
\label{asillysequel}
\mono_J(\ell) &\supset \Sp_{2g}(\ZZ / \ell \ZZ) \text{ for every prime number } \ell \geq 3.
\end{align}
Then, $[\grouphyp{g}{\QQ} : \mono_J] = 2$.
\end{lemma}
\begin{proof}
Since the maximal abelian extension $\QQ^{\on{ab}}$ is equal to the maximal cyclotomic extension $\QQ^{\on{cyc}}$, we have that
\begin{equation}\label{factcheck}
\rho_J(\Gal(\ol{\QQ}/\QQ^{\on{cyc}})) = \rho_J(\Gal(\ol{\bq}/\bq^{\ab})) = [\mono_J , \mono_J].
\end{equation}
Using~\eqref{factcheck}, we find that
\begin{align*}
[\grouphyp{g}{\QQ} : \mono_J] & = [\hyphat_{2g+2} : \rho_J(\Gal(\ol{\QQ}/\QQ^{\on{cyc}}))] \\
& = [\hyphat_{2g+2} : [\grouphyp{g}{\QQ},\grouphyp{g}{\QQ}]] \cdot [[\grouphyp{g}{\QQ}, \grouphyp{g}{\QQ}] : \rho_J(\Gal(\ol{\bq}/\bq^{\cyc}))] \\
& = 2 \cdot [[\hyphat_{2g+2}, \hyphat_{2g+2}] : [\mono_J, \mono_J]] ,
\end{align*}
where in the last step above we used the result of Corollary~\ref{theorem:r=2:special}. Thus, to prove that $H_J$ is maximal, it suffices to show that the inclusion $[\mono_J, \mono_J] \subset [\hyphat_{2g+2}, \hyphat_{2g+2}]$ is an equality.
The result then follows from Lemma~\ref{theorem:r=2}.
\end{proof}

Criterion~\eqref{asillysequel} may be broken down into two different sets of criteria by means of the following two propositions, adapted from~\cite{anni2017constructing} and~\cite{seaweed} respectively. The first set of criteria has the advantage that it implies the image is surjective
at all but finitely many primes,
although notably it does omit a finite list of primes $\ell$.

We first recall definitions from~\cite{anni2017constructing}.
\begin{definition}[\protect{\cite[Definition 1.2, Definition 1.3]{anni2017constructing}}]
	\label{definition:t-eisenstein}
	Let $t \geq 1$ be an integer and $p$ a prime. A polynomial $f(x) \defeq x^m + a_{m-1} x^{m-1} + \cdots + a_0 \in \mathbb Z_p[x]$
	is {\em $t$-Eisenstein} if $v(a_i) \geq t$ for $i > 0$ and $v(a_0) = t$, for $v$ the $p$-adic valuation.
	Further, suppose $q_1, \ldots, q_k$ are prime numbers and $f(x) \in \mathbb{Z}_p[x]$ is monic and squarefree.
	We say $f(x)$ is of {\em type} $t - \left\{ q_1, \ldots, q_k \right\}$ if it can be factored as
	$f(x) = h(x) \prod_{i=1}^k g_i(x-\alpha_i)$ over $\mathbb Z_p[x]$ for $\alpha_i \in \mathbb Z_p$ with
	$\alpha_i \not\equiv \alpha_j \bmod p$ for all $i \neq j$, $g_i(x)$ a $t$-Eisenstein polynomial
	of degree $q_i$, and $h(x) \bmod p$ a separable polynomial with $h(\alpha_i) \not \equiv 0 \bmod p$ for all $i$.
\end{definition}
The next proposition follows immediately upon combining the main results of~\cite{anni2017constructing}:
\begin{proposition}[\protect{\cite{anni2017constructing}}]
	\label{prop-anni}
Suppose $f \in \mathbb{Z}[x]$ satisfies the following properties:
\begin{enumerate}
\item There exist primes $q_1$, $q_2$, and $q_3$ such that
$$q_1 \leq q_2 < q_3 < q_1 + q_2 = 2g + 2.$$
\item There exist two distinct primes $p_{t_1}, p_{t_2} > g$ so that $f$ has type $1 - \left\{ 2 \right\}$ at $p_{t_1}$ and $p_{t_2}$.
\item There exists a prime $p_2 > 2g+2$ which is a primitive root modulo $q_1, q_2$, and $q_3$ so that $f$ has type $1 - \left\{ q_1, q_2 \right\}$ at $p_2 > 2g + 2$.
\item There exist a prime $p_3 > 2g+2$  which is a primitive root modulo $q_3$ such that $f$ has type $2-\left\{ q_3 \right\}$ at $p_3$.
\item Writing $f = x^{2g+2} + a_{2g+1} x^{2g+1} + \cdots + a_1x + a_0$ we have $a_0\equiv 2^{2g} \bmod 2^{2g+2}$,
	$a_{2g+1}\equiv 2 \bmod 2^{2g+2}$, and $a_i\equiv 0 \bmod 2^{2g+2-i}$ for $1 \leq i \leq 2g$.
\item For all primes $p \notin \left\{ 2,p_2,p_3 \right\}$ we have that $p^2 \nmid \disc f$, the discriminant of $f$.
\end{enumerate}
Let $J$ denote the Jacobian of the regular proper model for the affine curve $y^2 = f(x)$.
Then $\mono_J(\ell) \supset \Sp_{2g}(\ZZ / \ell \ZZ)$ for every $\ell > g$ so long as $\ell \not\in \{2,3,q_1, q_2, q_3, p_2, p_3\}$.
\end{proposition}
\begin{proof}
	We now demonstrate why Proposition~\ref{prop-anni} follows immediately from the results of \cite{anni2017constructing}.
	It suffices to verify the hypotheses of \cite[Theorem 6.2]{anni2017constructing}.
	Their hypotheses (G+$\varepsilon$), (2T), ($p_2$) and ($p_3$) are respectively (1), (2), (3), and (4) above.
	Next, we note that $f$ satisfies their condition (adm):
	We have to show $f$ is admissible (in the terminology of
	\cite[Definition 4.6]{anni2017constructing}) at all primes $p$.
	$f$ is admissible at $p_2$ by \cite[Lemma 4.10]{anni2017constructing} and at $p_3$ by \cite[Lemma 4.11]{anni2017constructing}.
	Further, $f$ is admissible at all primes with semistable reduction by \cite[Lemma 4.9]{anni2017constructing}
	so it suffices to show $f$ is semistable at all primes $p \notin \left\{ p_2, p_3 \right\}$.
	At $p = 2$ this follows from \cite[Lemma 7.7]{anni2017constructing}, using (5) above, while at
	all odd primes this follows from \cite[Lemma 7.5]{anni2017constructing}, using (6) above.
	To conclude the proof, we only need check that all primes $\ell > g$ with $\ell \not\in \{2,3,q_1, q_2, q_3, p_2, p_3\}$
	satisfy either \cite[Theorem 6.2(i) or (iii)]{anni2017constructing}.
	If $\ell \neq p_2, p_3$ then $\ell$ satisfies \cite[Theorem 6.2(i)]{anni2017constructing} because we have seen $J/\mathbb Q_\ell$
	is semistable above. If $\ell = p_2$ or $p_3$ then $\ell$ satisfies \cite[Theorem 6.2(iii)]{anni2017constructing} by (3) and (4) above.
\end{proof}

The second set of criteria has the advantage that it is simpler to state and works for every odd prime $\ell$.
The following criteria have, in essence, appeared in several papers including~\cite[Theorem 1.1]{renyaDW:classification-of-subgroups-of-symplectic-groups-over-finite-fields},~\cite[Theorem 1.1]{hall:big-symplectic-or-orthogonal-monodromy-modulo-l}, and~\cite[Proposition 2.2]{seaweed}.
\begin{proposition}
	\label{proposition:zywina-criterion}
	Let $g \ge 2$, and let $\ell \ge 3$ be prime. Consider a subgroup $H(\ell) \subset \GSp_{2g}(\mathbb{F}_\ell)$ satisfying the following conditions:
	\begin{enumerate}
		\item[\customlabel{propa}{(A)}] $H(\ell)$ contains a transvection, by which we mean an element with determinant $1$ that fixes a codimension-$1$ subspace.
		\item[\customlabel{propb}{(B)}] The action of $H(\ell)$ on $(\ZZ/\ell \ZZ)^{2g}$ is irreducible, in the sense that there are no nontrivial invariant subspaces.
\item[\customlabel{propc}{(C)}] The action of $H(\ell)$ on $(\ZZ / \ell \ZZ)^{2g}$ is primitive, in the sense that there does not exist a decomposition $(\ZZ / \ell \ZZ)^{2g} \simeq V_1 \oplus \dots \oplus V_k$ with $H(\ell)$ permuting the $V_i$'s.
	\end{enumerate}
	Then we have that $H(\ell) \supset \Sp_{2g}(\mathbb{Z} / \ell \ZZ)$.
\end{proposition}

\subsection{Checking the Criteria}\label{icheckthat}

The remainder of the paper is devoted to using the criteria introduced in Section~\ref{sweenytime} to verify the examples declared in Theorem~\ref{exemplinongratia}.

\subsubsection{Criterion~\eqref{thethingIwanttocheck}: The $2$-adic Component} \label{subsection:2-adic}

The following lemma allows us to verify criterion~\eqref{thethingIwanttocheck}:

\begin{lemma} \label{lemma:2-adic}
	Let $g \in \{2,3\}$, and let $H_2 \subset \hypaddict_{2g+2}$ be a closed subgroup. If $H_2(2) = S_{2g+2}$, then we have that $H_2 = \hypaddict_{2g+2}$.
\end{lemma}
\begin{proof}
	When $g = 2$, the inclusion $S_6 \subset \Sp_4(\bz / 2 \bz)$ is an equality, so the lemma follows from~\cite[Theorem 1]{landesman-swaminathan-tao-xu:lifting-symplectic-group}. For the rest of the proof, we take $g = 3$. Note that an easy generalization of the argument given in~\cite[Lemma 3, Section IV.3.4]{serre1989abelian} shows that, if $H \subset \Sp_{6}(\mathbb Z_2)$ is a closed subgroup satisfying $\hypaddict_{8}(8) = H(8)$, then $H_2 = \fy{2^{\infty}}{2}^{-1}(H(2))$.
So it suffices to show that
$\hypaddict_{8}(8) = H(8)$.
	Indeed, the following \texttt{magma} code verifies that there are no strict subgroups of
$\hypaddict_{8}(8)$ with $\bmod$-$2$ reduction equal to
$\hypaddict_{8}(2) = S_{8}$.
\vspace*{0.1in}
\begin{adjustwidth}{0.5in}{0in}
\begin{alltt}
G := GL(6,quo<Integers()|8>);
e := elt<G| 1,0,0,0,0,0,1,1,0,0,0,0,0,0,1,0,0,0,
\qquad 0,0,0,1,0,0,0,0,0,0,1,0,0,0,0,0,0,1>;
f := elt<G|1,1,0,0,0,0,0,1,1,0,0,0,1,1,1,1,0,0,
\qquad 1,1,0,1,1,0,1,1,1,1,1,1,1,1,1,1,0,1>;
H := sub<G|e,f>;
maximals := MaximalSubgroups(H);
grp, f := ChangeRing(G, quo<Integers()|2>);
for K in maximals do
    if #f(K{\`{}}subgroup) eq #H then
        assert false;
    end if;
end for; \qedhere
\end{alltt}
\end{adjustwidth}
\end{proof}

Recall from the discussion in Section~\ref{sec:single} that $\mono_J(2) = S_{2g+2}$ if and only if the polynomial $f(x)$ has Galois group $S_{2g+2}$. A simple {\tt magma} computation that this is the case for the polynomials $f(x)$ associated to the curves stated in Theorem~\ref{exemplinongratia}.
Then Lemma~\ref{lemma:2-adic} tells us that $(\mono_J)_2 = \hypaddict_{2g+2}$, thus verifying the criterion~\eqref{thethingIwanttocheck}.

\subsubsection{Criterion~\eqref{asillysequel}: The Genus-$2$ Example}

We now verify the genus-$2$ example. We first apply Proposition~\ref{prop-anni}. To verify the conditions (1)-(6) on the polynomial $f(x)$, we make the following choices:
$$q_1 = q_2 = 3,\, q_3 = 5,\, p_{t_1} = 3,\, p_{t_2} = 5,\, p_2 = 17,\, p_3 = 7.$$
Condition (1) is clearly satisfied and
conditions (2)-(4) are satisfied upon observing that $f(x)$ admits the following factorizations:
\begin{align*}
( x^4 + x^3 + x^2 + x + 1 )( x^2 - 3 ) & \bmod 3^2 \\
	( x^4 + x^2 + x + 1 )( x^2 - 5 ) & \bmod 5^2 \\
	( x^3 - 17 ) ( ( x-1 )^3 - 17 ) & \bmod 17^2 \\
	( x-1 )( x^5 - 7^2 ) & \bmod 7^3.
\end{align*}
Condition (5) is verified by reducing $f$ modulo $2^{2g+2} = 2^6 = 64$.
Finally, the computer verifies that the prime factorization of $\disc f$ is given by
\begin{equation*}
	\disc f = 3 \cdot 5 \cdot 7^8 \cdot 17^4 \cdot 421 \cdot 6397 \cdot 103434941173345262214445927 \cdot 4899652830439610728976665849.
	\end{equation*}
Hence, Proposition~\ref{prop-anni} tells us that condition~\eqref{asillysequel} holds for every odd prime $\ell$ satisfying $\ell \not\in \{3,5,7,17\}$.

To deal with the four remaining primes $\ell$, we utilize the criteria given in Proposition~\ref{proposition:zywina-criterion}.
First, we show the existence of a transvection (condition (A) of Proposition~\ref{proposition:zywina-criterion}).
Indeed, this follows from \cite[Lemma 2.9]{anni2017constructing}, which says that if there is a prime $p \nmid 2 \ell$ such that $f(x)$ has type $1 - \left\{ 2 \right\}$ when viewed as a polynomial in $\mathbb Z_p[x]$, then $J[\ell]$ contains a transvection.
For $\ell \in \left\{ 5,7,17 \right\}$ this follows by taking $p = 3$ while for $\ell =3$ this follows by taking $p = 5$.

To complete the proof, it suffices to verify conditions (B) and (C) of Proposition~\ref{proposition:zywina-criterion}.
For $p$ be a prime of good reduction of $J$, let $\frob_p \in G_{\QQ}$ denote the corresponding Frobenius element, and let $\charpoly_p(T) \in \ZZ[t]$ denote the characteristic polynomial of $\rho_J(\frob_p) \in \GSp_{2g}(\wh{\ZZ})$. The next proposition gives us a criterion to check irreducibility and primitivity together (conditions (B) and (C)):

\begin{proposition}[\protect{\cite[Proof of Lemma 7.2]{seaweed}}]\label{prop-irrop}
	Fix a prime $\ell \ge 3$. Suppose there exists $p \neq \ell$ of good reduction such that $\charpoly_p(T)$ is irreducible modulo $\ell$ and $\ell \nmid \operatorname{tr}(\frob_p)$. Then $H(\ell)$ acts irreducibly and primitively on $(\mathbb{F}_\ell)^{2g}$.
\end{proposition}

A simple {\tt magma} calculation shows that for $\ell \in \{3,17\}$, we can apply Proposition~\ref{prop-irrop} with
$$\on{ch}_{401}(T) = T^4 - 49T^3+1257T^2-19649T+160801.$$
Likewise, for $\ell = 5$, we can use
$$\on{ch}_{61}(T) = T^4 + 6T^3 + 54T^2 + 366T + 3721,$$
and for $\ell = 7$, we can use
$$\on{ch}_{277}(T) = T^4 + 31T^3 + 765T^2 + 8587T + 76729.$$
This completes the verification that the curve $C_2$ in Theorem~\ref{exemplinongratia} has maximal monodromy.

\subsubsection{Criterion~\eqref{asillysequel}: The Genus-$3$ Example}

We now verify the genus-$3$ example. We begin again by applying Proposition~\ref{prop-anni}. To verify the conditions (1)-(6) on the polynomial $f(x)$, we make the following choices:
$$q_1 = 3, \,  q_2 = 5,\, q_3 = 7,\, p_{t_1} = 5,\, p_{t_2} = 13,\, p_2 = 17,\, p_3 = 19.$$
Condition (1) is clearly satisfied and
conditions (2)-(4) are satisfied upon observing that $f(x)$ admits the following factorizations:
\begin{align*}
( x^6 + x^3 + x^2 + 1 )( x^2 + 5 ) & \bmod 5^2 \\
	 (x^6 + 51x^5 + 12 x^4 + 70x^3 + 82x^2 + 41x + 158)((x-10)^2 + 143(x-10)+78)  & \bmod 13^2 \\
	((x-1)^3 + 17)(x^5 + 17) & \bmod 17^2 \\
	(x+1)(x^7 + 361) & \bmod 19^3.
\end{align*}
Condition (5) is verified by reducing $f$ modulo $2^{2g+2} = 2^8 = 256$.
Finally, the computer verifies that the prime factorization of $\disc f$ is given by
\begin{align*}
	& \disc f = 2^{44} \cdot 5 \cdot 13 \cdot 17^6 \cdot 19^{12} \cdot 409 \cdot 71347 \cdot
249200273817326443 \cdot 2259862376409853901527 \cdot \\
& \qquad \qquad
\qquad 76378336963241484055881774103 \cdot 3700557180228322572272219236151
.
	\end{align*}
Hence, Proposition~\ref{prop-anni} tells us that condition~\eqref{asillysequel} holds for every odd prime $\ell$ satisfying $\ell \not\in \{3,5,7,13,17,19\}$.

To deal with the four remaining primes $\ell$, we again utilize the criteria given in Proposition~\ref{proposition:zywina-criterion}.
First, we show the existence of a transvection (condition (A) of Proposition~\ref{proposition:zywina-criterion}).
This follows from~\cite[Lemma 2.9]{anni2017constructing}, which says that if there is a prime $p \nmid 2 \ell$ such that $f(x)$ has type $1 - \left\{ 2 \right\}$ when viewed as a polynomial in $\mathbb Z_p[x]$, then $J[\ell]$ contains a transvection.
For $\ell \in \left\{ 3, 7,13,17,19 \right\}$ this follows by taking $p = 5$ while for $\ell =5$ this follows by taking $p = 13$.

To complete the proof, it suffices to verify conditions (B) and (C) of Proposition~\ref{proposition:zywina-criterion}. A simple {\tt magma} calculation shows that for $\ell = 3$, we can apply Proposition~\ref{prop-irrop} with
$$\on{ch}_{101}(T) = T^6 + 10T^5 + 60T^4 + 222T^3 + 6060T^2 + 102010T + 1030301.$$
Likewise, for $\ell = 5$, we can use
$$\on{ch}_{89}(T) = T^6 - 3T^5 + 93T^4 + 40T^3 + 8277T^2 - 23763T + 704969,$$
for $\ell \in \{7, 17\}$, we can use
$$\on{ch}_{127}(T) = T^6 - 12T^5 + 8T^4 + 548T^3 + 1016T^2 - 193548T + 2048383,$$
and for $\ell \in \{13,19\}$, we can use
$$\on{ch}_{103}(T) = T^6 - 7T^5 + 55T^4 - 191T^3 + 5665T^2 - 74263T + 1092727
.$$
This completes the verification that the curve $C_3$ in Theorem~\ref{exemplinongratia} has maximal monodromy.

\section{Conjecture Regarding Maximal Adelic Image of Hyperelliptic Curves}

As we saw in Section~\ref{symbed}, the $\bmod$-$2$ monodromy of a hyperelliptic
curve $y^2 = f(x)$ always lies in the subgroup $S_{2g+2} \subset \Sp_{2g}(\mathbb Z/2 \mathbb Z)$. Further, the $\bmod$-$2$ monodromy will be all of $S_{2g+2}$
if and only if the spitting field of $f(x)$ is as large as possible;
i.e., has Galois group $S_{2g+2}$ over the base field $K$.
In Lemma~\ref{lemma:2-adic}, we saw that for $g = 2$ or $3$,
if the $\bmod$-$2$ monodromy is surjective modulo $2$, then it is surjective
$2$-adically. We conjecture that this pattern continues to hold in higher
genera:
\begin{conjecture}
	\label{conjecture:surjective-mod-2-implies-surjective-2-adically}
	Let $g \geq 2$ and let $H \subset \hypaddict_{2g+2}$ be a closed
	subgroup. If $H(2) = S_{2g+2}$ then $H = \hypaddict_{2g+2}$.
\end{conjecture}
To conclude, we make some remarks on the consequences of this conjecture.
\begin{remark}
	\label{remark:}
	As described in the proof of Lemma~\ref{lemma:reduction-from-adelic-to-mod-l},
	via an easy generalization of the argument given
	in~\cite[Lemma 3, Section IV.3.4]{serre1989abelian},
	to prove
	Conjecture~\ref{conjecture:surjective-mod-2-implies-surjective-2-adically},
	it suffices to check $\hypaddict_{2g+2}(8) = H(8)$.
\end{remark}
\begin{remark}
	\label{remark:conjecture-implies-maximal-adelic}
Note that Conjecture~\ref{conjecture:surjective-mod-2-implies-surjective-2-adically}, if true, has the following useful consequence:
If $C$ is a hyperelliptic curve over $\mathbb Q$ with Jacobian $J$ satisfying
$H_J(2) = S_{2g+2}$ and $H_J(\ell) \supset \Sp_{2g}(\mathbb Z/\ell \mathbb Z)$
for every $\ell \geq 3$, then the $C$ has maximal adelic Galois image.
That is, $[\grouphyp{g}{\QQ} : \mono_J] = 2$.

Indeed, granting Conjecture~\ref{conjecture:surjective-mod-2-implies-surjective-2-adically}, this claim follows immediately from Lemma~\ref{lemma:reduction-from-adelic-to-mod-l}.
\end{remark}
\begin{remark}
	\label{remark:anni-maximal-adelic-examples}
	As follows from Remark~\ref{remark:conjecture-implies-maximal-adelic},
	Conjecture~\ref{conjecture:surjective-mod-2-implies-surjective-2-adically}, would imply that the examples of hyperelliptic
curves with maximal $\bmod$-$\ell$ image constructed in
\cite[Theorem 7.1]{anni2016residual} in fact have maximal adelic image.
\end{remark}

\section*{Acknowledgments}

\noindent This research was supervised by Ken Ono and David Zureick-Brown at the Emory University Mathematics REU and was supported by the National Science Foundation (grant number DMS-1557960). We would like to thank David Zureick-Brown for suggesting the problem that led to the present article and for offering us his invaluable advice and guidance.  We would like to thank
Samuele Anni
and
Vladimir Dokchitser
for helpful correspondence regarding applying their work.
We thank Jackson Morrow and David Zureick-Brown for reading early drafts and providing helpful feedback, and we thank the anonymous referee for providing a number of useful suggestions. We used {\tt Magma} and \emph{Mathematica} for explicit calculations.

\bibliographystyle{alpha}
\bibliography{bibfile}

\begin{thebibliography}{AdRDW16}

\bibitem[A'C79]{acampo:tresses-monodromie-et-le-groupe-symplectique}
N.~A'Campo.
\newblock Tresses, monodromie et le groupe symplectique.
\newblock {\em Comment. Math. Helv.}, 54(2):318--327, 1979.

\bibitem[AD17]{anni2017constructing}
S.~Anni and V.~Dokchitser.
\newblock Constructing hyperelliptic curves with surjective galois
  representations.
\newblock {\em arXiv preprint arXiv:1701.05915v1}, 2017.

\bibitem[AdRDW16]{renyaDW:classification-of-subgroups-of-symplectic-groups-over-finite-fields}
S.~Arias-de Reyna, L.~Dieulefait, and G.~Wiese.
\newblock Classification of subgroups of symplectic groups over finite fields
  containing a transvection.
\newblock {\em Demonstr. Math.}, 49(2):129--148, 2016.

\bibitem[ALS16]{anni2016residual}
S.~Anni, P.~Lemos, and S.~Siksek.
\newblock Residual representations of semistable principally polarized abelian
  varieties.
\newblock {\em Res. Number Theory}, 2:2:1, 2016.

\bibitem[CGJ11]{cojocaruGJ:one-parameter-families-of-elliptic-curves}
A.~C. Cojocaru, D.~Grant, and N.~Jones.
\newblock One-parameter families of elliptic curves over {$\Bbb Q$} with
  maximal {G}alois representations.
\newblock {\em Proc. Lond. Math. Soc. (3)}, 103(4):654--675, 2011.

\bibitem[CH05]{cojocaruH:uniform-results-for-serres-theorem-for-elliptic-curves}
A.~C. Cojocaru and C.~Hall.
\newblock Uniform results for {S}erre's theorem for elliptic curves.
\newblock {\em Int. Math. Res. Not.}, (50):3065--3080, 2005.

\bibitem[Die02]{dooleyfat}
L.~Dieulefait.
\newblock Explicit determination of the images of the {G}alois representations
  attached to abelian surfaces with {${\rm End}(A)=\Bbb Z$}.
\newblock {\em Experiment. Math.}, 11(4):503--512 (2003), 2002.

\bibitem[Duk97]{duke:elliptic-curves-with-no-exceptional-primes}
W.~Duke.
\newblock Elliptic curves with no exceptional primes.
\newblock {\em C. R. Acad. Sci. Paris S\'er. I Math.}, 325(8):813--818, 1997.

\bibitem[Gra00]{grant:a-formula-for-the-number-of-elliptic-curves-with-exceptional-primes}
D.~Grant.
\newblock A formula for the number of elliptic curves with exceptional primes.
\newblock {\em Compositio Math.}, 122(2):151--164, 2000.

\bibitem[Gre10]{greasy}
A.~Greicius.
\newblock Elliptic curves with surjective adelic {G}alois representations.
\newblock {\em Experiment. Math.}, 19(4):495--507, 2010.

\bibitem[Hal08]{hall:big-symplectic-or-orthogonal-monodromy-modulo-l}
C.~Hall.
\newblock Big symplectic or orthogonal monodromy modulo {$l$}.
\newblock {\em Duke Math. J.}, 141(1):179--203, 2008.

\bibitem[Hog82]{hog1982}
G.~M.~D. Hogeweij.
\newblock Almost-classical {L}ie algebras. {I}, {II}.
\newblock {\em Nederl. Akad. Wetensch. Indag. Math.}, 44(4):441--452, 453--460,
  1982.

\bibitem[Jon10]{josofabank}
N.~Jones.
\newblock Almost all elliptic curves are {S}erre curves.
\newblock {\em Trans. Amer. Math. Soc.}, 362(3):1547--1570, 2010.

\bibitem[LSTX17]{landesman-swaminathan-tao-xu:lifting-symplectic-group}
Aaron Landesman, Ashvin~A. Swaminathan, James Tao, and Yujie Xu.
\newblock Lifting subgroups of symplectic groups over {$\Bbb{Z}/\ell\Bbb{Z}$}.
\newblock {\em Res. Number Theory}, 3:Paper No. 14, 12, 2017.

\bibitem[LSTX19]{landesman-swaminathan-tao-xu:rational-families}
Aaron Landesman, Ashvin Swaminathan, James Tao, and Yujie Xu.
\newblock Surjectivity of {G}alois representations in rational families of
  abelian varieties.
\newblock {\em Algebra \& Number Theory}, 13(5):995--1038, 2019.
\newblock With an appendix by Davide Lombardo.

\bibitem[O'M78]{omeara1978symplectic}
O.~T. O'Meara.
\newblock {\em Symplectic groups}, volume~16 of {\em Mathematical Surveys}.
\newblock American Mathematical Society, Providence, R.I., 1978.

\bibitem[Ser72]{causalrelationship}
J.-P. Serre.
\newblock Propri\'et\'es galoisiennes des points d'ordre fini des courbes
  elliptiques.
\newblock {\em Invent. Math.}, 15(4):259--331, 1972.

\bibitem[Ser98]{serre1989abelian}
J.-P. Serre.
\newblock {\em Abelian {$l$}-adic representations and elliptic curves},
  volume~7 of {\em Research Notes in Mathematics}.
\newblock A K Peters, Ltd., Wellesley, MA, 1998.
\newblock With the collaboration of Willem Kuyk and John Labute, Revised
  reprint of the 1968 original.

\bibitem[Vas03]{vasiu2003surjectivity}
A.~Vasiu.
\newblock Surjectivity criteria for {$p$}-adic representations. {I}.
\newblock {\em Manuscripta Math.}, 112(3):325--355, 2003.

\bibitem[Wei96]{weigel:on-the-profinite-completion-of-arithmetic-groups-of-split-type}
Thomas Weigel.
\newblock On the profinite completion of arithmetic groups of split type.
\newblock In {\em Lois d'alg\`ebres et vari\'et\'es alg\'ebriques ({C}olmar,
  1991)}, volume~50 of {\em Travaux en Cours}, pages 79--101. Hermann, Paris,
  1996.

\bibitem[Yel15]{yelton2015thesis}
J.~Yelton.
\newblock {\em Hyperelliptic Jacobians and their associated l-adic Galois
  representations}.
\newblock PhD thesis, The Pennsylvania State University, 2015.

\bibitem[Zyw10a]{zywina2010elliptic}
D.~Zywina.
\newblock Elliptic curves with maximal {G}alois action on their torsion points.
\newblock {\em Bull. Lond. Math. Soc.}, 42(5):811--826, 2010.

\bibitem[{Zyw}10b]{zywina2010hilbert}
D.~{Zywina}.
\newblock {Hilbert's irreducibility theorem and the larger sieve}.
\newblock {\em arXiv:1011.6465v1}, November 2010.

\bibitem[{Zyw}15]{seaweed}
D.~{Zywina}.
\newblock {An explicit Jacobian of dimension 3 with maximal Galois action}.
\newblock {\em arXiv:1508.07655v1}, August 2015.

\end{thebibliography}

\end{document}